\documentclass[a4paper,10pt]{amsart}
\usepackage[utf8]{inputenc}
\usepackage[T1]{fontenc}
\usepackage{lmodern}
\usepackage{graphicx}
\usepackage[english]{babel}
\usepackage{subfigure}
\usepackage{mathrsfs}
\usepackage[bbgreekl]{mathbbol}
\usepackage{amsmath, amsfonts, amssymb, dsfont}
\usepackage{stmaryrd}
\usepackage{mathrsfs}
\usepackage{mathtools}
\usepackage{listings}
\usepackage{subfigure}
\usepackage{floatrow}
\usepackage{fourier-orns}
\usepackage{amsthm}
\usepackage{algorithm}
\usepackage{algorithmic}
\usepackage{textcomp}
\usepackage{faktor}
\usepackage{xfrac}
\usepackage[all]{xy}
\usepackage[colorlinks = True , allcolors = blue]{hyperref}

\newcommand{\tsp}[1]{{}^t \! #1}

\newcommand{\ii}{\mathrm{i}}

\newcommand{\Span}{\operatorname{Span}}

\newcommand{\im}{\operatorname{im}}
\newcommand*\diff{\mathop{}\!\mathrm{d}}

\newcommand{\Id}{\operatorname{Id}}

\newcommand{\R}{\mathbb R}
\newcommand{\N}{\mathbb N}

\newcommand{\CC}{\mathbb C}

\newcommand{\CCC}{\mathscr{C}}
\newcommand{\T}{\mathcal T}
\newcommand{\F}{\mathcal{F}}

\newcommand{\End}{\operatorname{End}}

\newcommand{\Ad}{\operatorname{Ad}}

\newcommand{\ad}{\operatorname{ad}}

\newcommand{\g}{\mathfrak{g}}

\newcommand{\Heis}{\operatorname{Heis}}
\newcommand{\heis}{\mathfrak{heis}}

\newcommand{\act}{\curvearrowright}

\newcommand{\lag}{\langle}
\newcommand{\rag}{\rangle}

\newcommand{\op}{\operatorname{Op}}
\newcommand{\ttt}{\mathfrak{t}}

\newcommand{\Sym}{\operatorname{Sym}}
\newcommand{\Cl}{\operatorname{Cl}}

\newcommand\isomto{\stackrel{\sim}{\smash{\longrightarrow}\rule{0pt}{0.4ex}}}

\def\dar[#1]{\ar@<2pt>[#1]\ar@<-2pt>[#1]}
\entrymodifiers={!!<0pt,0.7ex>+}

\newcommand{\eq}[1][r]
   {\ar@<-3pt>@{-}[#1]
    \ar@<-1pt>@{}[#1]|<{}="gauche"
    \ar@<+0pt>@{}[#1]|-{}="milieu"
    \ar@<+1pt>@{}[#1]|>{}="droite"
    \ar@/^2pt/@{-}"gauche";"milieu"
    \ar@/_2pt/@{-}"milieu";"droite"}
\everymath{\displaystyle\everymath{}}

\theoremstyle{plain}
\newtheorem{thm}{Theorem}[section]

\newtheorem{lem}[thm]{Lemma}
\newtheorem{prop}[thm]{Proposition}
\newtheorem{cor}{Corollary}[thm]
\newtheorem{ex}[thm]{Example}

\theoremstyle{definition}
\newtheorem{mydef}[thm]{Definition}

\theoremstyle{remark}
\newtheorem{rem}[thm]{Remark}

\DeclareSymbolFontAlphabet{\mathbb}{AMSb}
\DeclareSymbolFontAlphabet{\mathbbl}{bbold}

\title[]{Fields of Toeplitz algebras form the principal symbols of regular 2-step nilpotent groups}
\author{Clément Cren}
\address{Mathematisches Institut\\
Georg-August Universität Göttingen\\
Bunsenstraße 3-5\\
D-37073 Göttingen\\
Deutschland}
\email{\href{mailto:clement.cren@mathematik.uni-goettingen.de}{clement.cren@mathematik.uni-goettingen.de}}

\begin{document}

\begin{abstract}
    We show that the C*-algebra of a regular 2-step nilpotent Lie group can be recovered using continuous fields of Toeplitz algebras and a crossed product. We generalize this result to polycontact manifolds in the sense of van Erp which are endowed with fields of such groups. We also investigate those manifolds with a more rigid structure, namely those modeled on H-type groups. In all those cases, there is a certain pseudodifferential calculus named filtered calculus, we show that the algebra of principal symbols can also be recovered from the field of Toeplitz algebras.
\end{abstract}

\maketitle

\section{Introduction}

The study of certain hypoelliptic differential operators led to the introduction of a new pseudodifferential calculus, taking into account a certain anisotropy in the ambient space in the definition of its principal symbols. Here we will be interested in manifolds \(M\) endowed with a distribution \(H\subset TM\) such that:
\[\Gamma(H) + \left[ \Gamma(H),\Gamma(H)\right] = \Gamma(TM) .\]
Here \(\Gamma\) denotes the global sections of a vector bundle (the same statement works with compactly supported sections instead, the manifold \(M\) will be assumed compact anyway). The idea of this new calculus, following e.g. \cite{FollandStein,BealsGreiner}, is to consider vector fields in \(\Gamma(H)\) still as differential operators of order 1, but the other vector fields as having order 2. The principal symbol of an operator is then understood as a convolution operator on a family of graded nilpotent groups, here of step 2, the osculating groups. This family arises from the vector bundle \(\ttt_HM = H \oplus \faktor{TM}{H}\), endowed with the Lie bracket on sections given by the curvature form:

\begin{align*}
    B \colon \Gamma(H) \times \Gamma(H) &\to \Gamma\left(\faktor{TM}{H}\right) \\
                                  (X,Y) &\mapsto [X,Y]_{\Gamma(TM)} \mod \Gamma(H).
\end{align*}

We have \(B \in \Gamma(\Lambda^2H^* \otimes \faktor{TM}{H})\) so the Lie brackets can be localized on each fiber of \(\ttt_HM\) which is thus a family of Lie algebras, the osculating Lie algebras. These algebras being nilpotent, they can be integrated into a family of connected, simply connected Lie groups \(T_HM\) using the Baker-Campbell-Hausdorff formula, these groups are called the osculating Lie groups. This family of groups is in general not locally trivial as the group structure might change abruptly at a given point. The analysis of the operators in the calculus requires a good understanding of the algebra of principal symbols. The ellipticity condition for instance (called Rockland condition in this context, after \cite{Rockland}) requires invertibility of the principal symbols in all the nontrivial irreducible unitary representations of the osculating groups. Understanding these representations can be done through Kirillov theory \cite{Kirillov}. This theory provides a homeomorphism between the unitary dual of a nilpotent Lie group with the Fell topology and the set of co-adjoint orbits with the quotient topology. This calculus has gained a lot of attraction lately.

In this paper we assume that our groups have a specific structure, namely regular 2-step nilpotent groups to allow a good understanding of their unitary dual. Manifolds having all their osculating groups 2-step regular are called polycontact after \cite{VanErpPolycontact}. We show that the \(C^*\)-algebra of principal symbols (of order 0) on these manifolds can be constructed from a field of Toeplitz algebras. We give two different proofs of this identification. The first reduces the result to the case of Heisenberg group which was treated in \cite{CrenToeplitz}. The second relies on a finer analysis of symbols on regular step-2 groups, where we identify them with certain families of Weyl operators.

A particular example of regular 2-step nilpotent groups are the groups of Heisenberg type introduced in \cite{Kaplan}. These groups can be constructed from representations of Clifford algebras, in particular, the moduli space of their isomorphism classes is discrete. Because of that, we show that if a polycontact manifold has all its osculating groups of Heisenberg type then they are all isomorphic and the osculating groups form a locally trivial group bundle. In particular, the symbol algebra also becomes a locally trivial bundle of \(C^*\)-algebras over \(M\). These manifolds include contact manifolds, for which the result can be found in the author's previous work \cite{CrenToeplitz}, as well as quaternionic contact or octonionic contact manifolds \cite{BiquardQuaternionicContact}, and the cocompact quotients of any group of Heisenberg type.

\section{Toeplitz algebra of a symplectic vector space}


Given a complex hermitian vector space \(W\), we can construct a Toeplitz \(C^*\)-algebra \(\mathcal{T}(W)\) in natural way. Now, given a symplectic vector space \((V,\omega)\), we can choose a complex structure \(J\colon V \to V\) compatible with the symplectic form. We have then two complex hermitian vector spaces (of the same real dimension as \(V\)), \(V^{1,0},V^{0,1}\). The two algebras \(\T(V^{1,0}), \T(V^{0,1})\) are opposed to one another and depend naturally on the triple \((V,\omega,J)\) with respect to unitary symplectomorphisms. Replacing \(J\) by \(-J\) corresponds to changing the orientation of the symplectic vector space (i.e. replacing \(\omega\) by \(-\omega\) to keep compatibility) so it swaps the two algebras.

We describe this \(C^*\)-algebra as a subalgebra of bounded operators on a specific Hilbert space, namely the symmetric Fock space. For more details see \cite{CrenToeplitz}.
With the same notations as above, consider a complex hermitian vector space \(W\). We then take the positive symmetric tensor powers \(\Sym^k(W)\) with the induced hermitian metric. The symmetric Fock space is then given by the \(\ell^2\)-completion:
\[\mathcal{F}^+(W) = \ell^2-\bigoplus^{\perp}_{k\geq 0}\Sym^k(W).\]
If \(w\in W\), it induces a shift operator on \(\F^+(W)\) by:
\[S_w = \Sym(w \otimes \cdot).\]
The Toeplitz algebra is then generated by these operators and their adjoints:
\[\T(W) = C^*(S_w|w\in W) \subset \mathcal{B}(\F^+(W)).\]
This means that we consider all (non-commutative) polynomials between \(S_w\)'s and their adjoints, as well as the limits of such operators. If we fix an orthonormal basis \(e_1,\cdots,e_n\), then the Toeplitz algebra is generated by \(S_{e_1},\cdots,S_{e_n}\) and their adjoints. The commutation relations between these operators are then easier to express. In particular they commute modulo the compact operators. Considering their joint spectrum, we obtain the exact sequence:
\begin{equation}\label{Equation: ToeplitzHerm}
    \xymatrix{0 \ar[r] & \mathcal{K}(\F^+(W))\ar[r] & \T(W) \ar[r] & \CCC(\mathbb{S}^*W) \ar[r] & 0.}
\end{equation}
The quotient map is realized as follows. A vector \(w \in W\), seen as an element of \((W^*)^*\) defines a linear form on \(W^*\). By homogeneity, we can see this map as defined on the sphere \(\mathbb{S
}^*W\). We can also take the adjoints of such operators (we get the complex conjugate function), products of such operators (so homogeneous polynomial functions on \(W^*\)) and sums (arbitrary polynomials). This corresponds to the quotient map in the above exact sequence. This can be summarized as
\begin{equation}\label{Equation: image shift quotient}
    \forall w \in W, S_w \equiv w \mod \mathcal{K}(\F^+(W)).
\end{equation}

\begin{rem}
    This symbolic description of the Toeplitz algebra is due to Arveson \cite{Arveson}. Toeplitz operators can also be defined as multiplication operators on some spaces of (weighted) holomorphic \(L^2\)-functions (composed with the orthogonal projection onto the subspace of holomorphic functions). The resulting \(C^*\)-algebras are isomorphic.
    
    
    
\end{rem}

This algebra is compatible with unitary morphisms between the underlying vector spaces. Let \(W_1,W_2\) be two hermitian vector spaces, and \(T \colon W_1 \to W_2\) a unitary map. It induces a map between the symmetric tensor powers and hence between the symmetric Fock spaces:
\[\widetilde{T} \colon \F^+(W_1) \to \F^+(W_2)\]
which is also unitary.
We then have, for any \(w \in W\), the relation
\[\widetilde{T}\circ S_w = S_{T(w)}\circ \widetilde{T}.\]
Therefore the isomorphism \(\widetilde{T}\) intertwines the two Toeplitz algebras.

Consequently, if \(E \to M\) is a hermitian vector bundle, one can define a bundle of symmetric Fock spaces \(\F^+(E) \to M\). This is a locally trivial bundle of Hilbert spaces thanks to the unitary equivariance of the construction. We get corresponding bundle of Toeplitz algebras \((\T(E_m))_{m\in M}\) which is also locally trivial by the same remark. Denote by \(\T(E)\) the algebra of continuous sections of this bundle, vanishing at infinity if \(M\) is non-compact.
If \(X\in \Gamma_c(M,E)\) is a compactly supported section, we can build a section of \(\T(E)\) by \(S_X\), which on every fiber over \(m\in M\) takes the value \(S_{X(m)}\). The \(C^*\)-algebra \(\T(E)\) is generated by these sections.

Given a symplectic vector space \((V,\omega)\), we can take a compatible complex structure \(J \in \End(V)\). We get a hermitian vector space \(V^{1,0}\). We denote by 
\[\T(V,J) := \T(V^{1,0}),\]
the corresponding symmetric Fock space. The identification \(V \cong V^{1,0}\) identifies the respective (co)-spheres. This way we rewrite the exact sequence \ref{Equation: ToeplitzHerm} into:

\begin{equation}\label{Equation: ToeplitzSympl}
    \xymatrix{0 \ar[r] & \mathcal{K}(\F^+(V^{1,0}))\ar[r] & \T(V,J) \ar[r] & \CCC(\mathbb{S}^*V) \ar[r] & 0.}
\end{equation}

We obtain a similar construction for a symplectic vector bundle \(H \to M\) (i.e. there is a section \(\omega \in \Gamma(M,\Lambda^2H^*)\) which restricts to a symplectic form on each fiber). The space of compatible complex structures on a symplectic vector space is contractible. Therefore we may also choose a complex structure \(J\in \End(H)\) compatible with the symplectic form on each fiber. We then have a hermitian vector bundle \(H^{1,0} \to M\) and take the corresponding Toeplitz algebra (so the one given by sections of Toeplitz algebras over each fiber). Similarly as for symplectic vector spaces, we denote by
\[\T(H,J) := \T(H^{1,0}),\]
this Toeplitz algebra. The bundle of Hilbert spaces \(\F^+(H^{1,0})\) is a bundle of infinite dimensional, separable Hilbert spaces and is therefore trivializable, by a result of Dixmier and Douady \cite{DixmierDouady}. Consequently, the algebra of continuous sections vanishing at infinity \(\Gamma_0(M,\mathcal{K}(\F^+(H^{1,0})))\) is the algebra of sections trivializable field of compact operators. The algebra is thus isomorphic to \(\CCC_0(M,\mathcal{K})\). Therefore, this Toeplitz algebra fits into the exact sequence:

\begin{equation}\label{Equation: ToeplitzSymplBundle}
    \xymatrix{0 \ar[r] & \CCC_0(M,\mathcal{K}) \ar[r] & \T(H,J) \ar[r] & \CCC(\mathbb{S}^*H) \ar[r] & 0.}
\end{equation}

\section{Regular 2-step nilpotent groups and H-type groups}

In this paper we will work with the following kind of groups:

\begin{mydef}
    A regular 2-step nilpotent Lie algebra is a graded Lie algebra \(\mathfrak{g} = \g_1 \oplus \g_2\) such that \(\g_2 = \mathfrak{z}(\g) = [\g_1,\g_1]\) and moreover, for any \(x \in \g_1\), the map \(\ad_x \colon \g_1 \to \g_2\) is surjective.
    A regular 2-step nilpotent group is a connected, simply connected group, whose Lie algebra is regular 2-step nilpotent Lie algebra.
\end{mydef}

These groups are also called polycontact groups in \cite{VanErpPolycontact} or H-groups in the sense of Métivier (and not to be confused with the H-type groups later on). This class of groups was introduced in \cite{Metivier}. It was shown than on this group, hypoellipticity and analytic hypoellipticity were equivalent for homogeneous differential operators (so both equivalent to the Rockland condition). Here we show that the regularity condition gives a nice structure to the space of coadjoint orbits.

\begin{prop}\label{Proposition:Regular 2-step Conditions}
    Let \(\g = \g_1 \oplus \g_2\) be a 2-step nilpotent Lie algebra. The following are equivalent:
    \begin{enumerate}
        \item \(\g\) is regular.
        \item For every \(\theta\in \g_2^*\setminus0\), the form \(\omega_{\theta} \in \Lambda^2 \g_1^*\) defined by \(\omega_{\theta} \colon (u,v) \mapsto \theta([u,v])\) is non-degenerate (i.e. \((\g_1,\omega_{\theta})\) is a symplectic vector space).
    \end{enumerate}
\end{prop}

\begin{proof}
    If \(\g\) is not regular consider \(x\in \g_1\) for which \(\ad_x \colon \g_1 \to \g_2\) is not surjective. If \(\theta \in \g_2^*\setminus 0\) vanishes on \(\im(\ad_x)\), then \(\omega_{\theta}(x,\cdot) = 0\) and the form \(\omega_{\theta}\) is degenerate.

    If \(\g\) is regular, assume there is a \(\theta \in \g_2^*\) for which \(\omega_{\theta}\) has a kernel, and take \(x \in \ker(\omega_{\theta}) \setminus 0\). Let \(z \in \g_2\), by regularity, we can take \(y \in \g_1\) with \([x,y] = z\). Then \(\theta(z) = \omega_{\theta}(x,y) = 0\) so \(\theta = 0\) and \(\g\) satisfies \((2)\).
\end{proof}

We now compute the coadjoint orbits of a regular 2-step graded group. It will be used in the next section to describe its irreducible unitary representations using Kirillov theory.

\begin{prop}
    Let \(G\) be a regular 2-step graded group. Its coadjoint orbits consist of:
    \begin{itemize}
        \item Affine subspaces of the form \(\g_1^* \oplus \{\theta\}, \theta \in\g_2^* \setminus0\).
        \item Points \(\{\eta\}, \eta \in \g_1^*\).
    \end{itemize}
    Consequently we get a homeomorphism \(\faktor{\g^*}{G} \cong (\g_2^*\setminus0) \sqcup \g_1^*\). The topology on the right hand side being the euclidean one on each component, the first being open and the second closed. A net in \(\g_2^*\setminus0\) converging to \(0\) converges to all the points in \(\g_1^*\).
\end{prop}

\begin{proof}
    Let \(\eta \in \g_1^*, x \in \g\). The image of \(\ad_x\) is included in \(\g_2\), thus \(\ad^*_x(\eta) = 0\). Consequently, \(\forall g\in G, \Ad^*(g)(\eta) = \eta\) and \(\{\eta\}\) is a coadjoint orbit.

    Now consider \((\eta,\theta) \in \g^*, x \in \g_1\). We have \(\ad^*_x(\eta,\theta) = (\ad^*_x(\theta),0)\). The mapping \(x \mapsto \ad^*_x\theta\) from \(\g_1\) to its dual is the musical isomorphism induced by the form \(\omega_{\theta}\) and is in particular surjective. Therefore \(T_{(\eta,\theta)}\left(\Ad^*(G)\cdot(\eta,\theta)\right) = \g_1^*\). Consequently the coadjoint orbit going through \((\eta,\theta)\) is the affine space \(\g_1^* \oplus\{\theta\}\).
\end{proof}

\begin{rem}
    We see in the proof the importance of the regularity assumption. Without it we would still have the first kind of coadjoint orbits. However, the other kind of orbits would just be included in the affine spaces and not necessarily equal to them. These particular kind of orbits (called flat orbits) have been studied extensively since the work of Moore and Wolf \cite{MooreWolf}, as they were proved to correspond to square-integrable representations of the group (for any kind of nilpotent group).
\end{rem}

Now, using Kirillov theory, we can make these coadjoint orbits correspond to representations of the group. Given a general nilpotent group \(G\), Kirillov \cite{Kirillov} defines a one-to-one correspondence between the irreducible unitary representations of the group, and the co-adjoint orbits in the dual of its Lie algebra. This correspondence
\[\widehat{G} \isomto \faktor{\g^*}{\Ad^*(G)}\]
was also shown to be an homeomorphism (the spaces being endowed with the hull kernel topology and the quotient of the euclidean topology respectively). Kirillov also links certain properties and constructions of the representations with those of the co-adjoint orbits. The one we will use is the fact that the representation corresponding to an orbit of dimension \(2d, d\in \mathbb{N}\), can be realized on \(L^2(\R^d)\) (in a very specific way). This implies that the representations corresponding to \(0\)-dimensional co-adjoint orbits are characters and all the others are infinite dimensional.

Recall that for a locally compact group \(G\), we may form it's \(C^*\)-algebra \(C^*(G)\) (for more details and background on this construction, see e.g. \cite{DixmierBook}). First consider the algebra \(\CCC_c(G)\) with convolution
\[(f_1,f_2) \mapsto f_1\ast f_2 =\left(g\mapsto  \int_{\gamma \in G}f_1(\gamma)f_2(\gamma^{-1}g)\mathrm{d}\gamma\right).\]
The group \(C^*\)-algebra is then obtained by taking the closure of \(\CCC_c(G)\) acting by convolution on \(L^2(G)\). This is the reduced \(C^*\)-algebra, there is also a maximal \(C^*\) algebra but both coincide for amenable groups, in particular for nilpotent groups. We will only consider groups for which this is true, so that we may use properties of the maximal \(C^*\)-algebra as well.
Given a unitary representation \(\pi \colon G \to \mathcal{U}(H)\), we can extend it to a \(*\)-representation of \(C^*(G)\) by the formula:
\[\pi(f) = \int_{g\in G}f(g)\pi(g)\mathrm{d}g, f \in \CCC_c(G).\]
This formula extends to a bounded \(C^*\)-algebra homomorphism \(\pi \colon C^*(G) \to \mathcal{B}(H)\). Conversely, any non-zero algebra homomorphism of this sort comes from a unitary representation of \(G\). The irreducibility of such representations is also preserved and we get a one-to-one correspondence between unitary irreducible representations of the group \(G\) and irreducible representations of its \(C^*\)-algebra \(C^*(G)\).

\begin{rem}
    To define the \(C^*\)-algebra, we could also have started with \(\CCC^{\infty}_c(G)\) instead of all compactly supported continuous functions, the norm closure would stay the same. All the integrals considered here are with respect to a fixed Haar measure on the group but the resulting algebra doesn't depend on it up to isomorphism. Finally, this construction applies verbatim to smooth families of groups (more generally for Lie groupoids). That is, a pair of spaces with a surjective submersion \(G \to M\), and a smooth map \(m \colon G \times_M G \to G\) which restricts to a group law on each fiber (the inverse also needs to be smooth). We get a similar construction of \(C^*\)-algebra if we consider the group convolution fiberwise (the completion is also similar). This will be used in Section \ref{Section:Fields Toeplitz Algebras polycontact}. 
\end{rem}

We now go back to a regular 2-step graded group \(G\). Elements of \(\g_1^*\) are fixed points for the co-adjoint action so they represent characters of the group. In particular, \(0 \in \g_1^* \) corresponds to the trivial representation. The other representations, corresponding to some \(\theta \in \g_2^*\) are infinite dimensional.

\begin{cor}\label{Corollary: CStar Group is continuous field}
    Let \(G\) be a regular 2-step nilpotent group with Lie algebra \(\g = \g_1 \oplus \g_2\). Then \(C^*(G)\) is a continuous field of \(C^*\)-algebras over \(\g_2^*\). The fiber over a non-zero point is the algebra of compact operators over an infinite dimensional separable Hilbert space, \(\mathcal{K}\). The fiber over \(0\in \g_2^*\) is the commutative algebra \(\CCC_0(\g_1^*)\). The continuous field is trivial on \(\g_2^*\setminus 0\).
\end{cor}

\begin{proof}
    The map \(\widehat{G} \to \g_2^*\) which is the identity on \(\g_2^*\setminus 0\) and sends the rest to \(0\) is continuous and open. The result is then a direct consequence of a theorem of Lee \cite[Theorem 4]{Lee}. For the triviality away from 0, we need to show that if \(f \in \CCC^{\infty}_c(G)\), the function \(\theta \in\g_2^*\setminus 0 \mapsto \pi_{\theta}(f) \in \mathcal{K}\) is continuous. This will be clear in the next section where we give a concrete description of these representations.
\end{proof}

\begin{cor}\label{Corollary: Extension CStar Group}
    Let \(G\) be a regular 2-step nilpotent group with Lie algebra \(\g = \g_1 \oplus \g_2\). Then we have the extension of the group \(C^*\)-algebra:
    \[\xymatrix{0 \ar[r] & \CCC_0(\g_2^*\setminus 0,\mathcal{K}) \ar[r] & C^*(G) \ar[r] & \CCC_0(\g_1^*) \ar[r] & 0.}\]
\end{cor}

\begin{proof}
    The quotient map is the evaluation at zero in the previous continuous field description of \(C^*(G)\). The kernel of this map is exactly the sections that vanish at \(0 \in \g_2^*\). The field of compact operators being trivial away from \(0\) we obtain the corresponding ideal.
\end{proof}

\section{Representation theory for regular 2-step nilpotent groups}

The representation theory of regular 2-step nilpotent groups was already described in the last section through Kirillov theory. In this section we give a more concrete description of the representations, particularly the infinite dimensional ones. This description will allow a precise description for the symbol algebra through Weyl operators in the next section, giving the link with Toeplitz operators.

The representations corresponding to \(\g_1^*\) are characters so we focus on the other ones, which are infinite dimensional. Let \(\theta \in \g^*_2\setminus 0\). We obtain a symplectic form \(\omega_{\theta}\) on \(\g_1\). If we choose a compatible complex structure \(J\), we obtain a complex hermitian vector space \((\g_1,J)^{1,0}\) and consider the corresponding symmetric Fock space \(\F^+_{\theta,J} := \F^+((\g_1,J)^{1,0})\).

We construct the representation \(\pi_{\theta}\) of the Lie algebra \(\g\) by anti-self-adjoint, unbounded operators on \(\F^+_{\theta,J}\). This representation lifts to a representation of the group \(G\), as it satisfies the condition of Nelson \cite{Nelson} (we will write \(\pi_{\theta}\) for both the representation of the group and the Lie algebra). By Schur's Lemma, an irreducible representation has to map the center into scalar operators. In particular, it induces a linear form on \(\mathfrak{z}(\g) = \g_2\). This map is, following the construction of Kirillov, given by \(\theta \in \g_2\), i.e.
\[\pi_{\theta|\g_2} = \ii \theta .\] 
We then need to construct an operator for each element of \(\g_1\) such that their Lie bracket, which is an operator corresponding to an element of \(\g_2\), is one of the previous scalar ones.
Let \(X_1,\cdots,X_n,Y_1,\cdots,Y_n\) be a Darboux basis for \((\g_1,\omega_{\theta})\), compatible with \(J\), that is \(JX_j=-Y_j, JY_j = X_j, j\geq 1\). Consider then \(W_j = \frac{1}{\sqrt 2}(X_j+\ii Y_j)\). The vectors \(W_1,\cdots, W_n\) for an orthonormal basis for \((\g_1,J)^{1,0}\). This induces a basis on the symmetric Fock space given by the family of vectors \(\Sym(W_1^{\otimes \alpha_1} \otimes \cdots\otimes W_n^{\otimes \alpha_n}), \alpha \in \N^n\). We denote by \(|\alpha\rag\) the vectors obtained after normalization. The family \((|\alpha\rag)_{\alpha \in \N^n}\) is a Hilbert basis of the symmetric Fock space \(\F^+_{\theta,J}\).

Now the operators \(\pi_{\theta}(W_j), \pi_{\theta}(\overline{W_j})\) that we want to obtain have to satisfy the canonical commutation relations:
\[[\pi_{\theta}(W_j),\pi_{\theta}(\overline{W_j})] = \pi_{\theta}([W_j,\overline{W_j}]) = -\ii^2\omega_{\theta}(X_j,Y_j)= 1.\]
Similarly, if \(j\neq k\) the operators \(\pi_{\theta}(W_j)\) commutes with \(\pi_{\theta}(W_k)\) and \(\pi_{\theta}(\overline{W_k})\).
Consider the creation and annihilation operators on the Fock space \(A_j^*,A_j, 1\leq j \leq n\). They are unbounded self adjoint operators defined by:
\begin{align*}
    A_j|\alpha\rag &= \sqrt{\alpha_j}|\alpha_1,\cdots,\alpha_j - 1, \cdots\alpha_n\rag \\
    A_j^*|\alpha\rag &= \sqrt{\alpha_j+1}|\alpha_1,\cdots,\alpha_j + 1, \cdots\alpha_n\rag.
\end{align*}
We can also define these operators using the shifts of the Toeplitz algebra and the so-called number operator \(N\), which is densely defined, self-adjoint and defined by:
\[N|\alpha\rag = |\alpha||\alpha\rag, \quad |\alpha| = \sum_{j =1}^n\alpha_j, \alpha \in \N^n.\]
We then have the following relations, which can be deduced from direct computation (see \cite{Arveson,CrenToeplitz}):
\begin{equation}\label{Equation: Creation Annihilation Shifts}
    \forall 1\leq j \leq n, A_j^* = N^{1/2}S_{W_j} = S_{W_j} (N+1)^{1/2}.
\end{equation}

In this way, the creation operators can be seen as an unbounded version of the shifts defining the Toeplitz algebra (in the shift, some constant appear from the symmetrization making it a bounded operator). The main feature of these operators is that they satisfy the canonical commutation relations:
\[\forall 1\leq j,k\leq n, [A_j,A_k^*] = \delta_{j,k}.\]
We can therefore define
\begin{equation}\label{Equation: Representation creation annihilation}
    \pi_{\theta}(W_j) = \ii A_j^*, \pi_{\theta}(\overline{W_j}) = \ii A_j.
\end{equation}

This defines an irreducible representation of the Lie algebra \(\g\). To integrate this representation into a representation of the Lie group \(G\), Nelson's theorem \cite[Corollary 9.1]{Nelson} requires the operator \(\sum\nolimits_j\pi_{\theta}(V_j)^2\) (the \(V_j\) form a basis of \(\g\)) to be essentially self-adjoint on a dense linear subspace. This subspace has to be contained in all the domains of \(\pi_{\theta}(V)\) and \(\pi_{\theta}(V)\pi_{\theta}(W)\) where \(V,W\) are elements of \(\g\). Here this operator is \(-2N + C\mathrm{Id}\) where \(C \in \R\) is a constant. It is easily seen to be essentially self-adjoint on the dense subspace of finite linear combinations of vectors \(|\alpha\rag, \alpha \in \N^n\). The representation \(\pi_{\theta}\) can thus be integrated to a unitary representation of the group \(G\). As a representation of \(\g\), \(\pi_{\theta}\) is irreducible, thus it is still the case as a representation of the group \(G\). 

 Finally, we are interested in the corresponding representation of \(C^*(G)\), which we still denote \(\pi_{\theta}\). The image of \(C^*(G)\) by this representation is known to be the compact operators on \(\F^+_{\theta,J}\). It actually suffices to show that \(\pi_{\theta}(f)\) is compact for every \(f \in C^*(G)\). This follows from the following lemma by density.

 \begin{lem}
     If \(f \in \CCC_c^{\infty}(G)\), then \(\pi_{\theta}(f)\) compact.
 \end{lem}

 \begin{proof}
    Let \(D \in \mathcal{U}(\g)\) be an element of the universal enveloping algebra (basically, non-commutative polynomials in elements of \(\g\)) such that \(\pi_{\theta}(D) = N+1\), where \(N\) is the number operator as above. Let \(Z_{\theta} \in \g_2\) be an element with \(\theta(Z_{\theta}) = 1\). Then we can set:
    \[D = -\frac{1}{2}\left(\sum_{j=1}^n X_j^2+Y_j^2 \right)  - \ii (1-\frac{n}{2})Z_{\theta}. \]
    For \(f \in \CCC_c^{\infty}(G)\), we obtain
    \[\pi_{\theta}(f) = (N+1)^{-1}\pi_{\theta}(Df).\]
    Now the operator \(\pi_{\theta}(Df)\) is bounded since \(Df\) still has Schwarz decay. Moreover the operator \((N+1)^{-1}\) is compact thus \(\pi_{\theta}(f)\) is compact.
 \end{proof}

 \begin{rem}
     Notice that the same proof show that \(\pi_{\theta}(f)\) is trace class. It suffices to replace \(D\) by \(D^{n}\) because \((N+1)^{-n}\) is trace class (this follows from a quick computation since the operator \(N\) is diagonal in the basis \((|\alpha\rag)_{\alpha \in \N^n}\)). This argument of trace class even extends to elements \(f \in C^{\infty}(G)\) with rapid decay at infinity.
 \end{rem}

Recall that \(\g\) is endowed with a family of inhomogeneous dilations \(\delta_{\lambda}, \lambda > 0\). They act by Lie algebra automorphisms on \(\g\) so they also act on \(G\) and \(\widehat{G}\).

On the characters it becomes the usual dilation on \(\g_1^*\). On the infinite dimensional representations, we have: 
\[\delta_{\lambda}^*\pi_{\theta} =\pi_{\tsp{}\delta_{\lambda}\theta} = \pi_{\lambda^2\theta}.\]
This can be seen the following way. We have \(\omega_{\lambda^2\theta} = \lambda^2\omega_{\theta}\) by construction. Therefore we can take the same complex structure for both symplectic forms and then if \(X_1,\cdots,X_n, Y_1,\cdots, Y_n\) is a Darboux basis for \(\omega_{\theta}\) then \(\lambda X_1,\cdots, \lambda X_n, \lambda Y_1,\cdots, \lambda Y_n\) is a Darboux basis for \(\omega_{\lambda^2\theta}\). 

\begin{rem}[Polar coordinates parameterization] 
Using this homogeneity, we can parameterize the representations in polar coordinates by \(\R^*_+ \times \mathbb{S}^*\g_2\). This requires a choice of metric on \(\g_2^*\) (or equivalently, directly on \(\g_2\)) to embed the sphere into \(\g_2^*\). This choice is a priori arbitrary. In the case of H-type groups however, such a metric comes naturally in the group's definition.

If we don't make such a choice the bundle \(\g_1 \times \mathbb{S}^*\g_2\) only has a conformal symplectic structure (over a point \(\theta\) we have the conformal class \((t\omega_{\theta})_{t> 0}\)). The choice of an embedding \(\mathbb{S}^*\g_2 \to \g_2^*\setminus 0\) allows to (continuously) fix a representative for each of these classes.
\end{rem}

\begin{rem}[Changing orientation]We can use a similar reasoning to compare \(\pi_{\theta}\) and \(\pi_{-\theta}\). For the symplectic forms we have \(\omega_{-\theta} = -\omega_{\theta}\). Therefore we have to take the opposite complex structure. The corresponding hermitian vector space is then the same but with conjugate complex structure. For the symmetric Fock spaces we get \(\F_{-\theta,-J} = \overline{\F_{\theta,J}}\) and the representations switches the roles of \(X_j\)'s and \(Y_j\)'s.
\end{rem}

\section{Toeplitz algebra of a regular 2-step nilpotent group}


Let \(G\) be a 2-step regular nilpotent group. We fix a metric on \(\g_2\) and \((J_{\theta})_{\theta \in \mathbb{S}^*\g_2}\) a smooth family of complex structures on \(\g_1\) compatible with the forms \(\omega_{\theta}, \theta\in \mathbb{S}^*\g_2\). As explained before, the compatibility condition does not depend on the choice of \(\theta\) up to multiplication by a positive scalar. We may and will assume, that all the corresponding inner products \(\omega_{\theta}(\cdot,J_{\theta}\cdot)\) do not depend on \(\theta\). This is possible because a choice of inner product induces a compatible complex structure. Such a choice gives an unambiguous meaning to the spheres on \(\g_1\) and \(\g_1^*\).

For each \(\theta\in \mathbb{S}^*\g_2\) we can then form the Toeplitz algebra \(\T(\g_1,J_{\theta})\). Choosing local trivializations of the complex bundle then allows to glue these algebras to a locally trivial bundle of \(C^*\)-algebras. Indeed the choice of \(J\) reduces the structural group of the symplectic bundle \(\mathbb{S}^*\g_2 \times \g_1\) from \(Sp(2n,\R)\) to \(U(n)\) (here \(2n = \dim_{\R}(\g_1)\)) and the Toeplitz algebra of a hermitian vector space is invariant under unitary transformations of that vector space. We denote by \(\T(\g,J)\) the \(C^*\)-algebra of continuous sections of that bundle. 

\begin{rem}
    If we replace \(J\) by another complex structure \(J'\), we get two complex structures on the bundle \(\mathbb{S}^*\g_2\times \g_1\). The two complex structures are compatible with the same symplectic structure so the corresponding complex vector bundles are homeomorphic. However, this homeomorphism might not preserve the symplectic structure. Indeed if it did, it would preserve the hermitian metrics induced by the symplectic form and the complex structures. Already at the level of a single symplectic vector space, we see that not all complex vector space isomorphisms are unitary. 
\end{rem}

For each \(\theta\in \mathbb{S}^*\g_2\), the symplectic space on the fiber gives a Toeplitz short exact sequence with ideal the compact operators and quotient continuous functions on the co-sphere of \(\g_1\). 
\[\xymatrix{0 \ar[r] & \mathcal{K}(\F^+((\mathfrak{g}_1, J_{\theta})^{1,0}))\ar[r] & \T(\g_1,J_{\theta}) \ar[r] & \CCC(\mathbb{S}^*\g_1) \ar[r] & 0.}\]
Although the hermitian bundle structure induced by \(J\) on \(\mathbb{S}^*\g_2\times \g_1\) is not trivial, the corresponding co-sphere bundle only depends on the real structure as we can obtain it as \(\mathbb{S}^*\g_2\times \left(\faktor{\g_1^*\setminus 0}{\R^*_+}\right) = \mathbb{S}^*\g_2\times \mathbb{S}^*\g_1\). We can therefore glue the quotient algebras and obtain \(\CCC(\mathbb{S}^*\g_2\times \mathbb{S}^*\g_1)\) seen as a continuous field of \(C^*\)-algebras over \(\mathbb{S}^*\g_2\). The corresponding ideal is a locally trivial bundle of compact operators on the bundle of symmetric Fock spaces of \((\mathcal{F}^+((\g_1,J_{\theta})^{1,0}))_{\theta\in \mathbb{S}^*\g_2}\). This bundle of Hilbert spaces is infinite dimensional, hence trivializable by a result of Dixmier and Douady \cite{DixmierDouady}, providing an isomorphism between the ideal and \(\CCC(\mathbb{S}^*\g_2)\otimes\mathcal{K}\). This discussion gives the following result:

\begin{prop}
    We have the exact sequence:
    \[\xymatrix{0 \ar[r] & \CCC(\mathbb{S}^*\g_2)\otimes\mathcal{K} \ar[r] & \T(\g,J) \ar[r] & \CCC(\mathbb{S}^*\g_2\times \mathbb{S}^*\g_1) \ar[r] & 0}\]
\end{prop}

Our goal is to relate this bundle of Toeplitz algebras to the \(C^*\)-algebra of the group. Ultimately this should yield a Morita equivalence between our algebra and \(C_0^*(G)\rtimes \R^*_+\) where \(C^*_0(G)\subset C^*(G)\) is the kernel of the trivial representation. This algebra has a similar short exact sequence but the quotient is \(\CCC(\mathbb{S}^*\g_1)\) so the algebra \(\T(\g,J)\) is too big. 

\begin{mydef}
    Let \(J\) be a compatible complex structure on the symplectic vector bundle \(\mathbb{S}^*\g_2\times \g_1\). We denote by \(\T_0(\g,J) \subset \T(\g,J)\) the subalgebra of elements for which the image in the quotient \(\CCC(\mathbb{S}^*\g_2\times \mathbb{S}^*\g_1)\) is constant on the \(\mathbb{S}^*\g_2\) direction.
\end{mydef}

Let us give a more concrete description of this subalgebra. Each element of \(X\in\g_1\) provides a continuous section of the complex vector bundle \((\mathbb{S}^*\g_2 \times (\g_1, J)^{1,0})\) by 
\[X_c \colon \theta \mapsto \frac{1}{\sqrt{2}}(X -\ii J_{\theta}X).\] 
From this, we create a family of shifts \(S_X\colon \theta\mapsto S_{X_c(\theta)} \in \T(\g_1,\theta)\). By construction, their image on the quotient is the constant map \(\theta \mapsto X \in (\g_1^*)^*\cong\g_1\) using Equation \ref{Equation: image shift quotient} and the identification of real vector spaces \(\g_1\cong (\g_1,J_{\theta})^{1,0}\) (this identification is the same one we used to go from \(X\) to \(X_c\) above). If we only consider these shifts however, we will only get constant sections of the bundle of compact operators. Therefore we obtain the following result:

\begin{prop}
    With the same notations as before \(\T_0(\g,J)\) is the subalgebra of sections of \(\T(\g,J)\) generated by the sections \(S_X, X\in \g_1\) and sections of the bundle of \(C^*\)-algebras \(\left(\mathcal{K}\left(\mathcal{F}^+((\g_1,J_{\theta})^{1,0})\right)\right)_{\theta\in \mathbb{S}^*\g_2}\).
\end{prop}

\begin{proof}
    The second subalgebra is a subalgebra of \(\T_0(\g,J)\). For elements of the form \(S_X, X\in \g_1\) this results from the discussion above. Sections of the bundle of compact operators have zero image in the quotient, in particular it is constant. Now this second subalgebra have the same ideal and quotient so it is equal to \(\T_0(\g,J)\).
\end{proof}

\begin{ex}
    If \((V,\omega)\) is a symplectic vector space, let \(\heis(V,\omega)\) be the corresponding Heisenberg Lie algebra obtained by seeing \(\omega\) as a 2-cocycle on the abelian group \(V\). We have \(\g_2 = \R\), thus \(\mathbb{S}^*\g_2 = \{\pm 1\}\). The symplectic vector space corresponding to \(+1\) is \((V,\omega)\) and the one to \(-1\) is \((V,-\omega)\).
    Then, if \(J \colon V \to V\) is a complex structure compatible with \(\omega\), \(-J\) is compatible with \(-\omega\) and we have \((V,-J)^{1,0} = (V,J)^{0,1}\). We will then write \(V^{1,0} = (V,J)^{1,0}\) and \(V^{0,1}\) accordingly. The Toeplitz algebra of the group is then a fibered product:
    \[\T_0(\heis(V,\omega), J) = \T(V^{1,0})\bigoplus_{\CCC(\mathbb{S}^*V)}\T(V^{0,1}).\]
    This means that an element in \(\T_0(\heis(V,\omega), J)\) corresponds to elements \((T_+,T_-) \in \T(V^{1,0}) \oplus \T(V^{0,1})\), such that the images of \(T_{\pm}\) modulo compact operators are the same elements in \(\CCC(\mathbb{S}^*V)\).
    This Toeplitz algebra is the one considered in \cite{CrenToeplitz}.
\end{ex}

Remember that the algebra \(\T_0(\g,J)\) acts on the bundle \(\F^+((\g_1,J)^{1,0})\). On this bundle, we can define a number operator \(N\). This is an unbounded operator which, on the subbundle \(\Sym^k((\g_1,J)^{1,0}), k \in \N,\) is given by multiplication by \(k\). The operator \(N+1\) is then positive. We can consider the complex powers \((N+1)^{\ii t/2}, t\in \R\), which then form a continuous family of unitary operators.
We can consider the conjugation \(\Ad((N+1)^{\ii t/2})\) as a group of automorphisms of \(\mathcal{B}(\F^+((\g_1,J)^{1,0}))\). This action of \(\R\), that we denote by \(\alpha\), restricts fiberwise. It is proved in \cite[Corollary 4.2.1]{CrenToeplitz} that it preserves the Toeplitz algebra, and is trivial on the quotient. Therefore we have an action \(\R \act \T(\g,J)\). Since the action is trivial on the quotient, then sections that have constant image in the quotients are also preserved so we also get an action \(\R \act \T_0(\g,J)\). The key relation is that for \(X \in \Gamma((\g_1,J)^{1,0})\), then:
\[\alpha_t(S_X) \equiv S_X \mod \CCC(\mathbb{S}^*\g_2,\mathcal{K}).\]

\begin{prop}
    We have the exact sequence:
    \[\xymatrix{0 \ar[r] & \CCC_0(\g_2^*\setminus \{0\})\otimes \mathcal{K} \ar[r] & \T_0(\g,J)\rtimes_{\alpha} \R \ar[r] & \CCC_0(\g_1^*\setminus \{0\}) \ar[r] & 0}.\]
\end{prop}
\begin{proof}
    As said before the action is trivial on the quotient so:
    \begin{align*}
        \CCC(\mathbb{S}^*\g_1)\rtimes_{\alpha}\R &\cong \CCC(\mathbb{S}^*\g_1) \otimes C^*(\R) \\
                                                 &\cong \CCC(\mathbb{S}^*\g_1) \otimes \CCC_0(\R^*_+) \\
                                                 &\cong \CCC_0(\g_1^*\setminus\{0\}).
    \end{align*}
    For the ideal, notice that \((N+1)^{\ii t/2}\) is a multiplier of \(\CCC(\mathbb{S}^*\g_2,\mathcal{K})\). The crossed product is then isomorphic to a trivial one and we get the result by similar computations.
\end{proof}

The exact sequence of that proposition looks like the one of Corollary \ref{Corollary: Extension CStar Group} when the trivial representation (which corresponds to \(0 \in \g^*_1\)) is removed. If we denote by \(C^*_0(G) \triangleleft C^*(G)\) the kernel of the trivial representation, we should expect an isomorphism \(\T_0(\g,J) \rtimes_{\alpha}\R \cong C^*_0(G)\). This is proved in \cite{CrenToeplitz} when \(G\) is a Heisenberg group. We generalize this result to all the 2-step regular nilpotent groups. We do this by reducing to the case of Heisenberg groups (but for families).

For a symplectic vector space \((V,\omega)\), recall from Corollary \ref{Corollary: CStar Group is continuous field} that the group \(C^*\)-algebra \(C^*(\Heis(V,\omega))\) is a continuous field of \(C^*\)-algebras over \(\R\). If we fix a compatible complex structure \(J\), the fiber at \(t \in \R\) is identified with the compact operators \(\mathcal{K}(\F^+((V,\mathrm{sign(t)}J)^{1,0}))\) if \(t \neq 0\) and \(\CCC_0(V^*)\) if \(t = 0\). Denote by \(I_+(V,\omega)\) the quotient of \(C^*(\Heis(V,\omega))\) by the sections vanishing for negative values (this corresponds to restricting the sections from \(\R\) to \(\R_+\)).

Because \(G\) is regular, we have a symplectic bundle \(\g_1 \times \mathbb{S}^*\g_2 \to \mathbb{S}^*\g_2\) with the forms \(\omega_{\theta}, \theta \in \mathbb{S}^*\g_2\). Consider the  bundle of \(C^*\)-algebras \((I_+(\g_1,\omega_{\theta}))_{\theta\in \mathbb{S}^*\g_2}\) and denote by \(\widetilde{C}^*(G)\) the algebra of continuous sections of this bundle. We have the exact sequence:
\[\xymatrix{0 \ar[r] & \CCC_0(\g_2^*\setminus 0,\mathcal{K}) \ar[r] & \widetilde{C}^*(G) \ar[r] & \CCC_0(\g_1^*\times \mathbb{S}^*\g_2) \ar[r] & 0.}\]
Here the quotient should be seen as \(\CCC(\mathbb{S}^*\g_2, \CCC_0(\g_1^*))\). It corresponds to the field of quotients corresponding to the characters for each group.

Consider \(\widetilde{C}^*_0(G)\subset \widetilde{C}^*(G)\) to be the subalgebra of sections that vanish at the points \((0_{\R},0_{\g_1^*})\) (i.e. whose image in the quotient in the above exact sequence, belongs to \(\CCC(\mathbb{S}^*\g_2, \CCC_0(\g_1^*\setminus0))\)). 

\begin{thm}
    The algebra \(C^*(G)\) is isomorphic to the subalgebra of \(\widetilde{C}^*(G)\) corresponding to sections \(F \in C^*(G)\) for which the images in the commutative quotients \(F({\theta})(0)\in \CCC_0(\g_1^*)\) do not depend on \(\theta \in \mathbb{S}^*\g_2\).
    
    Under this identification, the algebra \(C_0^*(G)\) corresponds to the elements in the image of \(C^*(G)\) that are also in \(\widetilde{C}^*_0(G)\).
\end{thm}

\begin{proof}
    Consider \(f \in C^*(G)\). We see it as a section of a continuous field of \(C^*\)-algebras over \(\g_2^*\) as in Corollary \ref{Corollary: CStar Group is continuous field}. Then to \(\theta \in \mathbb{S}^*\g_2\) we map the restriction of \(f\) to \(\R_+\theta\):
    \[F(\theta)(t) = f(t\theta), t\geq 0.\]
    This provides an element \(F(\theta) \in I_+(\g_1,\omega_{\theta})\) with a continuous dependence in \(\theta\). The image in the commutative quotients for each \(\theta\) is then given by \(f(0)\) which doesn't depend on \(\theta\). 

    This mapping is a continuous \(*\)-algebra morphism \(C^*(G) \to \widetilde{C}^*(G)\). Conversely, every element of \(F\) which has constant commutative quotients is of this form.
\end{proof}

\begin{thm}\label{Theorem: Toeplitz Field crossed product}
    Let \(G\) be a regular 2-step nilpotent group. Let \(J\) be a compatible complex structure on the symplectic vector bundle \(\mathbb{S}^*\g_2\times \g_1\). We have an isomorphism:
    \[\T(\g,J) \rtimes_{\alpha} \R \isomto \widetilde{C}_0^*(G).\]
    This identification induces an isomorphism:
    \[\T_0(\g,J)\rtimes_{\alpha} \R \cong C^*_0(G).\]
    The last isomorphism is \(\R^*_+\)-equivariant for the dual action on the left and \(\delta^*\) on the right.
\end{thm}

\begin{proof}
    We will use \cite[Theorem 5.9]{CrenToeplitz}, which says that if \((V,\omega)\) is a symplectic vector space, then \(\T(V^{1,0})\rtimes \R\) is isomorphic to \(I_{+,0}(V,\omega)\), where \(I_{+,0}(V,\omega)\) is the kernel of the map \(f \mapsto f(0_{\R})(0_{V^*})\). Moreover this isomorphism is \(\R^*_+\)-equivariant and also \(U(n)\)-invariant. Therefore it generalizes to symplectic vector bundles (with fixed compatible complex structure).

    We now apply this result to \(\widetilde{C}^*_0(G)\) which is the subalgebra of sections of the bundle \((I_{+,0}(\g_1,\omega_{\theta}))_{\theta\in \g_2^*\setminus 0}\). This gives the first part of the theorem. This isomorphism induces the following diagram:
    \[\xymatrix{\T(\g,J)\rtimes_{\alpha} \R \ar[r] \ar[d] & \CCC_0(\mathbb{S}^*\g_2\times(\g_1^*\setminus \{0\})) \ar@{=}[d]\\
    \widetilde{C}^*_0(G) \ar[r] & \CCC_0(\mathbb{S}^*\g_2\times(\g_1^*\setminus \{0\}))}\]
    If we take on both sides the pre-image of \(\CCC_0(\g_1^* \setminus 0)\) (i.e. the sections that have constant commutative quotient) then we get the second part of the theorem.
\end{proof}


\section{Fields of Weyl algebras}

In the previous section, we have reduced the proof to the case of the Heisenberg group to use the result from \cite{CrenToeplitz}. Another possibility is to adapt directly the proof for the Heisenberg group to the case of an arbitrary 2-step regular nilpotent group. This is the purpose of this section. Given an element \(T \in \mathcal{T}_0(\g_1,J)\), we want to construct a family of multipliers of \(C^*(G)\) seen as a continuous field over \(\g_2^*\) by taking \(F_T(r\theta) = T_{\theta}, r> 0, \theta\in \mathbb{S}^*\g_2\). The previous section also gives that \(F_T(0)\) is the multiplication by \(T_{\theta}\mod \mathcal{K} \in \CCC(\mathbb{S}^*\g_1)\) which does not depend on \(\theta\) by construction of \(\mathcal{T}_0(\g_1,J)\). The only problem with that approach is to understand the continuity at \(0 \in \g_2^*\) of this field of multipliers. To understand this, we will see the Toeplitz operators as a particular case of Weyl operators. The families of Weyl operators then arise from a pseudodifferential calculus, naturally associated to the group \(G\) and its filtration, called the filtered calculus. We do not need a full understanding of this calculus as what appears here are the principal symbols of such operators (which, as in the case of classical pseudodifferential calculus, can be understood as "local models" for the operators themselves). In particular, we will only define principal symbols and their product, leaving out the question of how to quantize them into actual operators, even though the symbols can already be seen as operators themselves. The interested reader can check e.g. \cite{FermanianFischer,vanErpYunckenPseudo,CrenFiltered} and their references. 

Principal symbols of order \(k \in \CC\) in the filtered calculus for the group \(G\) are smooth functions on \(\g^*\setminus 0\), homogeneous of degree \(k\) for the dilations \(\delta_{\lambda}, \lambda> 0\). These functions multiply in a way that takes the group structure into account. By homogeneity, such a function \(\sigma\) is characterized by:
\begin{itemize}
    \item a function on \(\g_1^*\setminus 0 \), homogeneous of degree \(k\) (now for the usual dilations), hence a function \(\sigma_0\) on the sphere \(\mathbb{S}^*\g_1\).
    \item a smooth family of functions \(\sigma_{\theta}\) on \(\g_1^*\), parameterized by \(\theta \in \mathbb{S}^*\g_2\).
\end{itemize}
If we get two symbols \(a \in \Sigma^k(G), b \in \Sigma^{\ell}(G)\), they multiply the following way: \((ab)_0 = a_0b_0\) is the usual (commutative) product of functions, and \((ab)_{\theta} = a_{\theta}\#_{\theta}b_{\theta}\). The product \(\#_{\theta}\) is constructed using the symplectic form \(\omega_{\theta}^*\) on \(\g_1^*\) and is non-commutative. This product is given by the integral formula, where \(2n = \dim(\g_1)\),
\begin{equation}\label{Equation: Weyl product}
    \forall u \in \g_1^*, a_{\theta}\#_{\theta}b_{\theta}(u) = \frac{1}{\pi^{2n}}\int_{v,w\in \g_1^*}e^{2\mathrm{i}\omega_{\theta}(v,w)}a_{\theta}(u+v)b_{\theta}(u+w)\mathrm{d}v\mathrm{d}w.
\end{equation}

As is, this product is only defines for \(a_{\theta}\) and \(b_{\theta}\) having enough decay at infinity (e.g. if they are in \(\mathscr{S}(\g_1^*)\)). We can see it as a composition formula of symbols for the Weyl quantization. This will give its sense to the integral formula above. To this end, notice like in \cite{EpsteinLectures,vanErpGorokhovsky}, that the homogeneity of \(a\) on the subspace \(\g_1^* \oplus \R \cdot \theta\) gives that \(a_{\theta}\) is an element of the class \(S^m_{\mathrm{iso}}(\g_1^*)\) and admits the polyhomogeneous expansion at infinity:
\[a_{\theta}(v) \sim \sum_{k \geq 0}a_{\theta}^{(k)}(v)\|v\|^{m-2k}\]
with \(a_{\theta}^{(k)}\) homogeneous of degree \(m-2k\), and \(\|\cdot\|\) any norm on \(V\) (we can take the one obtained from a choice of a complex structure \(J_{\theta}\) compatible with \(\omega_{\theta}^*\)).
The class of functions \(S^m_{\mathrm{iso}}(\g_1^*)\) is similar to the Hörmander classes but here, no separation of variables into \(\R^n \times \R^n\) is considered. For \(a_{\theta}\) this means that for every multi-index \(\alpha\):
\begin{equation}\label{Equation: Isotropic norm estimate}
    \exists C_{\alpha}\geq 0, \forall u\in \g_1^*, |\partial_u^{\alpha}a_{\theta}(u)|\leq C_{\alpha}(1+\|u\|)^{m-|\alpha|}.
\end{equation}
The polyhomogeneous expansion comes from the fact that the sphere at infinity of \(\g_1^*\times \{\theta\}\subset \g_1^* \oplus \R\cdot\theta\) is \(\mathbb{S}^*\g_1\times \{0\}\). The expansion is then just a Taylor expansion for the function \(a_{|\mathbb{S}(\g_1^* \oplus \R\cdot\theta)}\) near the equator. In particular the first term of this extension is \(a_{|\mathbb{S}^*\g_1\times \{0\}} = a_{0}\), independent of \(\theta\). This is reinterpreted below in terms of Weyl principal symbol. 

The norm estimates \ref{Equation: Isotropic norm estimate} give sense to the composition formula \ref{Equation: Weyl product}, as we even have:
\[\forall k,\ell \in \mathbb{C}, S^k_{\mathrm{iso}}(\g_1^*)\#_{\theta} S^{\ell}_{\mathrm{iso}}(\g_1^*) \subset S^{k+\ell}_{\mathrm{iso}}(\g_1^*).\]

\begin{rem}
    The references \cite{EpsteinLectures,vanErpGorokhovsky} only treat the case of the Heisenberg group. The situation here is however very similar and can actually be reduced to the Heisenberg group using the following trick. 
    If we fix \(\theta \in \mathbb{S}^*\g_2\), we restricted function on \(\g^*\) to \(\g_1^* \oplus \R\cdot \theta\). This subspace is equal to \(\mathfrak{h}_{\theta}^{\perp} \subset \g^*\), where \(\mathfrak{h}_{\theta} :=\ker(\theta)\subset \g\) is an ideal. The space \(\mathfrak{h}_{\theta}^{\perp}\) is stable under the co-adjoint action, and elements from the subgroup \(H_{\theta}\) corresponding to \(\mathfrak{h}_{\theta}\) act trivially. Under Kirillov correspondence, the coadjoint orbits of \(\mathfrak{h}_{\theta}^{\perp}\) are the co-adjoint orbits of the quotient group \(G/H_{\theta}\). This group is isomorphic to \(\mathrm{Heis(\g_1,\omega_{\theta})}\).

    Finally, on the symbol side, we can push-forward symbols by submersions between Lie groups. Here this gives:
    \[\Sigma^k(G) \to \Sigma^k(\Heis(\g_1,\omega_{\theta})).\]
    This push-forward corresponds, by Kirillov's correspondence, to the restriction of functions on \(\g^*\) to the subspace \(\mathfrak{h}_{\theta}^{\perp}\).

    In particular, the function \(a_{\theta}\) above, can be computed by first restricting \(a\) to \(\mathfrak{h}_{\theta}^{\perp}\), where the results of \cite{EpsteinLectures,vanErpGorokhovsky} apply. 
\end{rem}

We now describe the product in terms of composition of operators. Let \(L_{\theta}\oplus L_{\theta}' = \g_1^*\) be a decomposition into a pair of complementary lagrangians for \(\omega_{\theta}^*\). If we have a compatible complex structure \(J\), we can take any lagrangian subspace \(L_{\theta}\) and set \(L'_{\theta} = JL_{\theta}\). We can then quantize the symbol \(a_{\theta}\) into a Weyl operator \(\op^W(a_{\theta}) \colon \mathscr{S}(L_{\theta}) \to \mathscr{S}(L_{\theta})\), by the formula:
\[\forall f\in \mathscr{S}(L_{\theta}), \forall x\in L_{\theta},\op^W(a_{\theta})(f)(x) = \frac{1}{(2\pi)^n}\int_{\substack{y\in L_{\theta}\\ \xi \in L_{\theta}'}}a_{\theta}\left(\frac{x+y}{2},\xi\right)f(y)\mathrm{d}y\mathrm{d}\xi.\]

A straightforward computation then gives:
\[\op^W(a_{\theta})\circ \op^W(b_{\theta})=\op^W(a_{\theta}\#_{\theta}b_{\theta}).\]
The formula for the product \(\#_{\theta}\) shows that this algebra of operators does not depend on the choice of lagrangians. We thus denote by \(\mathcal{W}^k(\g_1^*,\omega_{\theta}^*)\) the vector space \(S^k_{\mathrm{iso}}(\g_1^*)\) which we now see with the product \(\#_{\theta}\) satisfying:
\[\#_{\theta} \colon \mathcal{W}^k(\g_1^*,\omega_{\theta}^*) \times \mathcal{W}^{\ell}(\g_1^*,\omega_{\theta}^*) \to \mathcal{W}^{k+\ell}(\g_1^*,\omega_{\theta}^*), \quad k,\ell \in \mathbb{C}.\]
Given \(a_{\theta}\in \mathcal{W}^k(\g_1^*,\omega_{\theta}^*)\), taking the principal part in the polyhomogeneous expansion \(a_{\theta}^{(0)}\) gives a continuous map which we will call the Weyl principal symbol:
\[\sigma_{W}^k\colon \mathcal{W}^k(\g_1^*,\omega_{\theta}^*) \to \CCC^{\infty}(\mathbb{S}^*\g_1).\]
This map is compatible with the product of Weyl operators. If \(a_{\theta}\in \mathcal{W}^k(\g_1^*,\omega_{\theta}^*)\) and \(b_{\theta}\in \mathcal{W}^{\ell}(\g_1^*,\omega_{\theta}^*)\) then 
\[\sigma_W^{k+\ell}(a_{\theta}\#_{\theta}b_{\theta}) = \sigma_W^k(a_{\theta})\sigma_W^{\ell}(b_{\theta}).
\]
As mentioned above with the polyhomogeneous expansion, in the case where the operator \(a_{\theta}\) comes from an element \(\Sigma^k(G)\), the Weyl principal symbol is given by \(\sigma^k_W(a_{\theta}) =a_0\). Therefore the Weyl principal symbol of the elements \(a_{\theta}\) does not depend on \(\theta\).

\begin{ex}
    Consider \(X \in \g_1\) or \(X\in \g_2\), seen as the principal symbol of a differential operator of order \(1\) or \(2\) respectively. They define (linear) functions \(\ell_X\) on \(\g^*\) by duality. Since \(X\) is of degree \(1\) or \(2\), this function is homogeneous for the dilations and is thus an element of \(\Sigma^1(G)\) or \(\Sigma^2(G)\) respectively.

    Notice that in the second case \(X\in \g_2\), the function \(\ell_X\) vanishes on the characters. This means that \((\ell_{X})_0=0\). This can also be seen from the polyhomogeneous expansion in the other representations. Let \(\theta \in \mathbb{S}^*\g_2\). The operator induced on the representation is \((\ell_{X})_{\theta} = \mathrm{i}\theta(X)\mathrm{Id}\). Since the symbol \(\ell_X\) has order 2, we have \((\ell_{X})_{\theta} \in S^2_{\mathrm{iso}}(\g_1^*)\). The first term in the homogenous expansion is thus quadratic. Since \((\ell_{X})_{\theta}\) is bounded this means that the principal part of its polyhomogeneous expansion, hence \((\ell_X)_0\), vanishes.
\end{ex}

Combining what is said above we get the following. The spaces of Weyl operators \(\mathcal{W}^k(\g_1^*,\omega_{\theta}^*)\) form a smooth bundle of Fréchet spaces over \(\mathbb{S}^*\g_2\). We denote its sections by \(\mathcal{W}^k(\g^*)\). We can thus see principal symbols \(a \in \Sigma^k(G)\) as smooth sections of the field of algebras of Weyl operators \((\mathcal{W}^k(\g_1^*,\omega_{\theta}^*))_{\theta \in \mathbb{S}^*\g_2}\), such that the principal symbol is constant. Of course not all such fields arise from elements in \(\Sigma^k(G)\) because of the gluing condition on these operators on \(\mathbb{S}^*\g_1\) only gives the continuity and not smoothness (see \cite{vanErpGorokhovsky} for such a discussion in the case of the Heisenberg group). At the level of \(C^*\)-completion for operators of order \(0\) it is however enough.

The estimate \ref{Equation: Isotropic norm estimate} shows that Weyl operators of order \(0\) extend continuously to operators on \(L^2(L_{\theta})\) (for a lagrangian splitting of \((\g_1^*,\omega^*_{\theta})\) as previously). We denote by \(W(\g_1^*,\omega_{\theta}^*)\) the \(C^*\)-algebra obtained by closure of \(\mathcal{W}^0(\g_1^*,\omega_{\theta}^*)\) for the operator norm in \(L^2(L_{\theta})\). This \(C^*\)-algebra does not depend on the choice of lagrangian splitting. The principal symbol extends continuously to a \(C^*\)-algebra homomorphism to \(\CCC(\mathbb{S}^*\g_1)\). We get the exact sequence:
\begin{equation}
    \xymatrix{0\ar[r] & \mathcal{K}(L^2(L_{\theta})) \ar[r] & W(\g_1^*,\omega_{\theta}^*) \ar[r]^{\sigma_W^0} & \CCC(\mathbb{S}^*\g_1) \ar[r] & 0.}
\end{equation}
The ideal \(\mathcal{K}(L^2(L_{\theta}))\) here appears as the completion of \(\mathcal{W}^{-1}(\g_1^*,\omega^*_{\theta})\subset \mathcal{W}^0(\g_1^*,\omega^*_{\theta})\).

We define \(W_0(\g)\) the \(C^*\)-algebra of sections of the field \((W(\g_1^*,\omega_{\theta}^*))_{\theta \in \mathbb{S}^*\g_2}\) obtained by completion of \(\mathcal{W}\)
\begin{thm}\label{Theorem: Symbols are fields of weyl operators}
    Let \(\Sigma(G)\) be the completion of \(\Sigma^0(G)\) in the multipliers of \(\ker(\pi_0) \subset C^*(G)\). The map \(\Sigma^0(G) \to \mathcal{W}^0(\g^*)\) extends to an embedding \(\Sigma(G) \to W(\g^*)\) whose image is exactly \(W_0(\g^*)\). Consequently we have an isomorphism \(\Sigma(G) \cong W_0(\g^*)\).
\end{thm}

\begin{proof}
    Let \(\sigma \in \Sigma^0(G)\). We have 
    \[\|\sigma\| = \sup_{\pi \in \widehat{G}\setminus 1}\|\pi(\sigma)\| = \sup_{\theta\in \mathbb{S}^*\g_2}\|\pi_{\theta}(\sigma)\|.\]
    The reduction from \(\widehat{G}\setminus 1\) to \(\mathbb{S}^*\g_2\) is done using homogeneity and density of the representations \(\pi_{\eta}, \eta \in \g_2^*\setminus 0\) in \(\widehat{G}\).
    We can realize the representation \(\pi_{\theta}\) the following way. Take an element \(f \in \CCC_c^{\infty}(G)\) and consider its euclidean Fourier transform \(\widehat{f} \in \CCC^{\infty}(\g^*)\). We can then restrict it to the co-adjoint orbit \(\g_1^*\times \{\theta\}\) and see it as a Weyl symbol (of order \(-\infty\) because it is compactly supported). To get actual operators on a Hilbert space, we can take a pair of transverse lagrangian subspaces for \(\omega_{\theta}^*\), \(\g_1^* = L_{\theta} \oplus L_{\theta}'\). We have
    \[\pi_{\theta}(f) = \op^W\left(\widehat{f}_{|\g_1^*\times \{\theta\}}\right)\in \mathcal{K}(L^2(L_{\theta})).\]
    The operator \(\pi_{\theta}(\sigma)\) is then represented by \(\op^W(\sigma_{\theta})\). This shows that the map \(\Sigma^0(G) \to \mathcal{W}^0(\g^*)\) is isometric, hence continuous, so we can extend it to the respective closures.

    We now have an isometric embedding
    \[\iota \colon \Sigma(G) \to W(\g^*).\]
    By construction, its image is included in \(W_0(\g^*)\). The algebra \(W_0(\g^*)\) sits in the exact sequence:
    \[\xymatrix{0 \ar[r] & \CCC(\mathbb{S}^*\g_2,\mathcal{K}) \ar[r] & W_0(\g^*) \ar[r]^{\sigma^0_W} & \CCC(\mathbb{S}^*\g_1) \ar[r] & 0.}\]
    The map \(\iota\) is compatible with the quotient maps in the sense that \(\sigma^0_W\circ \iota \colon \sigma \mapsto \sigma_0\). Therefore we have to show that the induced map from the kernel of \(\sigma \mapsto \sigma_0\) to \(\CCC(\mathbb{S}^*\g_2,\mathcal{K})\) is onto (hence an isomorphism). We will denote by \(I(G) \subset \Sigma(G)\) the kernel of \(\sigma \mapsto \sigma_0\) (these symbols are the symbols of the so-called Hermite-type operators, hence the notation). Both \(I(G)\) and \(\CCC(\mathbb{S}^*\g_2,\mathcal{K})\) are \(\CCC(\mathbb{S}^*\g_2)\)-algebras and the map induced by \(\iota\) preserves that structure. We can thus check that the induced map between the fibers over \(\theta\) is an isomorphism for each \(\theta \in \mathbb{S}^*\g_2\). For a fixed \(\theta\) this map corresponds to the same map for the group \(\mathrm{Heis}(\g_1, \omega_{\theta})\) where it is known to be an isomorphism.
\end{proof}

On the other hand, given a symplectic vector space \((V,\omega)\) with compatible complex structure \(J\), we can identify \(\T(V,J)\) with \({W}(V^*,\omega^*)\).
To this end, remember that \(\T(V,J)\) is generated by shift operators \(S_1,\cdots,S_d\) and their adjoints. These shifts are obtained by taking an orthonormal basis \(W_1,\cdots,W_d\) of \(V^{1,0}\) and setting \(S_j = S_{W_j}\). Now we can also see \(W_1,\cdots,W_d\) as homogeneous functions on \(V^*\), hence elements in \(\mathcal{W}^1(V^*,\omega^*)\). Similarly for their conjugates \(\overline{W_j}\).

These elements act as \(\ii \) times the annihilation operators, and the \(\overline{W_j}\) as \(\ii\) times the creation operators by \ref{Equation: Representation creation annihilation}. This can also be seen by direct computation. Let us write \(W_j = \frac{1}{\sqrt{2}}(X_j+\ii Y_j), \quad JX_j =-Y_j, 1\leq j \leq d\) and take the lagrangian subspaces \(L = \Span(Y_j^*, 1\leq j \leq d)\subset V^*\) and \(L' = J^*L = \Span(X_j^*, 1\leq j \leq d)\). For a function \(f\in \mathscr{S}(L)\) we obtain:
\begin{align*}
    \op^W(Y_j) f(y) &= \frac{1}{(2\pi)^d}\iint_{x,\xi\in \R^d} e^{\ii\omega^*(y-x,\xi)} X_j(\frac{x+y}{2},\xi)f(y)\mathrm{d}x\mathrm{d}\xi \\
    &=\frac{1}{(2\pi)^d}\iint_{x,\xi\in \R^d} e^{\ii \lag\xi,y-x\rag}\frac{x_j+y_j}{2}f(y)\mathrm{d}x\mathrm{d}\xi \\
    &= \int_{x\in\R^d}\delta_{x=y}\frac{x_j+y_j}{2}f(y)\mathrm{d}x \\
    &=y_jf(y).
\end{align*}
Similarly, \(\op^W(X_j) =- \ii \frac{\partial}{\partial y_j}\). This gives:
\[\op^W(W_j) = \frac{i}{\sqrt{2}}\left( x_j - \frac{\partial}{\partial x_j} \right).\]
The operators \(\frac{1}{\sqrt{2}}\left( x_j - \frac{\partial}{\partial x_j} \right)\) and \(\frac{1}{\sqrt{2}}\left( x_j + \frac{\partial}{\partial x_j} \right)\) are then well known representatives of the creation and annihilation operators on \(L^2(\R^d)\).

If we now want to represent the shifts of the Toeplitz algebra, we can use the relation \ref{Equation: Creation Annihilation Shifts} rewritten as:
\[\forall 1\leq j \leq d, S_j = A_j^* (N+1)^{-1/2}.\]
Recall that we can also recover \(N\) from the \(A_j^*\) so here we have:
\begin{equation}\label{Equation: Number +1 is Weyl operator}
    \op^W\left(-\sum_{j=1}^d W_j\overline{W_j} -(d-1)\right) = N+1.
\end{equation}
The operator \(N+1\) is positive and invertible. Since \(-\sum_{j=1}^d W_j\overline{W_j} -(d-1) \in \mathcal{W}^2(V^*, \omega^*)\), we may find \(\eta\in\mathcal{W}^{-1}(V^*,\omega^*)\) such that \(\op^W(\eta) = (N+1)^{-1/2}\).

\begin{lem}\label{Lemma: Embedding Toeplitz into Weyl + symbols}
    There is an embedding \(\T(V,J) \to W(V^*,\omega^*)\) sending \(S_j\) to \(-\ii W_j\eta\) and \(S_j^*\) to \(-\ii \eta \overline{W_j}\). Moreover, the composition with the Weyl principal symbol \(\sigma^0_W\) is the quotient map for Toeplitz operators their compact ideal.
\end{lem}
\begin{proof}
    We have for every \(1\leq j \leq d\), \(-\ii W_j\eta,-\ii \eta \overline{W_j} \in \mathcal{W}^0(V^*,\omega^*)\). Moreover their image by \(\op^W\) are \(S_j\) and \(S_j^*\) respectively by construction. This shows that the \(C^*\)-subalgebra \(W'\) generated by these operators is isomorphic to \(\T(V,J)\). All that remains to show is \(W' = W(V^*,\omega^*)\). Consider the Weyl principal symbol restricted to \(W'\). If \(z_1,\cdots,z_d\) denote the coordinates in the basis \(W_1,\cdots,W_d\) we have
    \[\sigma_W^1(W_j)(z) = z_j,\quad \sigma^1_W(\overline{W_j})(z) = \overline{z_j},\]
    as well as
    \[\sigma^2_W\left(-\sum_{j=1}^d W_j\overline{W_j} -(d-1)\right)(z) = |z|^2.\]
    By multiplicativity of the principal symbol we obtain:
    \[\sigma^0_W(-\ii W_j\eta)(z) = -\ii\frac{z_j}{|z|},\quad \sigma^0_W(-\ii \eta\overline{W_j})(z) = -\ii\frac{\overline{z_j}}{|z|}.\]
    Under the identification \(\CCC^{\infty}(\mathbb{S}^*V) \cong \CCC^{\infty}(\mathbb{S}^*V^{1,0})\), these symbols are the images of \(S_j\) and \(S_j^*\) in the by quotient of \(\T(V,J)\) by compact operators as in \ref{Equation: image shift quotient}.
\end{proof}

We want this embedding to be an isomorphism. We will use a 2 out of 3 lemma thanks to the exact sequences of \(\T(V,J)\) and \(W(V^*,\omega^*)\) respectively. The quotient was treated in the previous lemma. We now need to show that the algebra \(W'\) contains all compact operators on \(L^2(\R^d)\). It suffices to show that the representation of the Canonical Commutation Relations constructed above is irreducible, i.e. that the intersection of the kernels of the annihilation operators is 1-dimensional. If this is the case, fix a nonzero vector in this intersection, then the vectors 
\[\op^W(\ii\overline{W_{j_1}})\cdots \op^W(\ii\overline{W_{j_k}})v_0, 1\leq j_1,\cdots,j_k\leq d, k \geq 0,\]
span a dense subspace of \(L^2(\R^d)\).
\begin{lem}
    The realization of the Canonical Commutation Relations on \(L^2(\R^d)\) given by \(\ii \op^W(W_j)\) and \(\ii \op^W(\overline{W_j})\) for \(1\leq j \leq d\) is irreducible. That is 
    \[\dim\left(\bigcap_{j = 1}^d \ker(\op^W(\ii \overline{W_j}))\right) = 1.\]
\end{lem}
\begin{proof}
    Since \[\ii \op^W(\overline{W_j}) = \frac{1}{\sqrt{2}}\left(x_j + \frac{\partial}{\partial x_j}\right)\]
    the intersection of the kernels is the line directed by 
    \[x \mapsto \exp\left(-\frac{|x|^2}{2}\right).\]
\end{proof}

Remember that the choice of complex structure gives us an action of \(U(n)\) on \(V\), preserving \(\omega\) and \(J\). Likewise, we get an action on \(V^*\) preserving the dual objects. Elements in \(\mathcal{W}^k(V^*,\omega^*)\) functions on \(V^*\), and the estimates \ref{Equation: Isotropic norm estimate} are invariant by the \(U(n)\)-action. Therefore, \(U(n)\) acts on \(\mathcal{W}^k(V^*,\omega^*)\). Moreover, this action is compatible with the product of Weyl operators, which can be seen in Equation \ref{Equation: Weyl product}. In particular, the \(C^*\)-algebra \(W(V^*,\omega^*)\) is endowed with an action of \(U(n)\).

\begin{cor}\label{Corollary: Toeplitz isom Weyl}
    The embedding \(\T(V,J) \to W(V^*,\omega^*)\) sending \(S_j\) to \(-\ii W_j\eta\) and \(S_j^*\) to \(-\ii \eta \overline{W_j}\) is onto, i.e. it is an isomorphism. More precisely we have the commutative diagram:
    \[\xymatrix{0\ar[r] & \mathcal{K}(L^2(\R^d)) \ar[r] & W(V^*,\omega^*) \ar[r]^{\sigma_W^0} & \CCC(\mathbb{S}^*V) \ar[r] & 0 \\
    0 \ar[r] & \mathcal{K}(\F^+(V^{1,0}))\ar[r] \ar[u] & \T(V,J) \ar[r]  \ar[u]  & \CCC(\mathbb{S}^*V) \ar[r]  \ar@{=}[u]  & 0.}\]
    Finally, this isomorphism doesn't depend on the choice of orthonormal basis for \(V\), i.e. is equivariant under the respective \(U(n)\)-actions on both spaces.
\end{cor}

\begin{proof}
    By irreducibility of the representation of the Canonical Commutation Relations, we can construct a unitary isomorphism \(U\colon \F^+(\mathbb{C}^d) \to L^2(\R^d)\) intertwining both representations of the Canonical Commutation Relations. The map \(\T(V,J) \to W(V^*,\omega^*)\) then becomes the conjugation by \(U\). This preserves the ideal of compact operators in both spaces. We thus get the commutative diagram as claimed. The respective quotient maps have been identified in Lemma \ref{Lemma: Embedding Toeplitz into Weyl + symbols}. Since the map on ideals and quotients are isomorphisms, then the map \(\T(V,J) \to W(V^*,\omega^*)\) is an isomorphism as well.

    To check the \(U(n)\)-invariance, we just have to check that \(\eta\) is invariant by the \(U(n)\)-action. This follows from the fact that \(N+1\) is \(U(n)\)-invariant, which is immediate from Equation \ref{Equation: Number +1 is Weyl operator}.
\end{proof}
 
We can then extend this idea to fields of Weyl operators to get a new description of \(\Sigma(G)\) for a  2-step regular nilpotent group \(G\).

\begin{thm}\label{Theorem: Symbols group}
    Let \(G\) is a 2-step regular nilpotent group, \(\Sigma(G)\) the \(C^*\)-algebra of principal symbols of order \(0\) in the filtered calculus on \(G\). Let \(J\) be a family of almost complex structures on \(\g_1\) compatible with the symplectic forms. We then have an isomorphism:
    \[\Sigma(G) \cong \T_0(\g,J).\]
\end{thm}

\begin{proof}
    We fix a compatible complex structure on the symplectic bundle \(\g_1\times \mathbb{S}^*\g_2 \to \mathbb{S}^*\g_2\).
    We have two continuous fields of \(C^*\)-algebras. On the one hand we have \((W(\g_1^*,\omega_{\theta}^*))_{\theta\in \mathbb{S}^*\g_2}\). On the other hand we have \((\T(\g_1,J_{\theta}))_{\theta\in \mathbb{S}^*\g_2}\). By the corollary above, the fibers of the two fields are isomorphic. Moreover, these fields are actually bundles of \(C^*\)-algebras. This is because we can locally trivialize \(\g_1\times \mathbb{S}^*\g_2 \to \mathbb{S}^*\g_2\) as a \(U(n)\)-bundle. In this local trivialization, the isomorphism between the fibers induces a continuous bundle map. This implies that the notion of continuity for sections is the same for both fields. Hence the respective algebras of sections are isomorphic. Now we identified \(\Sigma(G)\) with the sections of \((W(\g_1^*,\omega_{\theta}^*))_{\theta\in \mathbb{S}^*\g_2}\) for which the principal Weyl symbol is constant in Theorem \ref{Theorem: Symbols are fields of weyl operators}. This subalgebra of section is then isomorphic to the subalgebra of sections \(\T_0(\g,J)\) by the identification of quotient maps in Lemma \ref{Lemma: Embedding Toeplitz into Weyl + symbols}. Therefore \(\Sigma(G) \cong \T_0(\g,J)\).   
\end{proof}

This identification gives another proof of Theorem \ref{Theorem: Toeplitz Field crossed product} by then applying Theorem 5.10 in \cite{CrenFiltered}. This theorem shows that if we transfer our \(\R\)-action on \(\T_0(\g,J)\) to \(\Sigma(G)\) through the isomorphism of the previous theorem, then the corresponding map
\[\Sigma(G)\rtimes\R \to C^*_0(G)\]
is an isomorphism. Proposition 5.6 of the same paper also gives another explanation on why elements of \(\T_0(\g,J)\) act as bounded multipliers of \(C^*_0(G)\). Indeed it is a general fact for graded groups that elements in \(\Sigma(G)\) act as bounded multipliers of \(C^*_0(G)\).

\begin{rem}
    Take \(\g\) as previously, with family of complex structures \(J\) corresponding to an inner product \(|\cdot|_{\g_1}\). We write the same norm for the one induced on the dual space.
    
    As a Weyl operator on \((\g_1^*, \omega_{\theta}^*)\), we have 
    \[N+\frac{n}{2} = \op^W(-\frac{1}{2}|\cdot|^2_{\g_1}).\]
    This does not depend on \(\theta\). In particular, we can extend this to every \(\theta\) by setting:
    \[\sigma = -\frac{1}{2}|\cdot|^2_{\g_1} \in \CCC^{\infty}(\g^*\setminus0 ).\]
    This function is homogeneous of order 2 and satisfies \(\pi_{\theta}(\sigma) = N+\frac{n}{2}\) for every \(\theta\in \mathbb{S}^*\g_2\). In particular \(\sigma \in \Sigma^2(G)\).

    In order to get \(N+1\) precisely, one could subtract a multiple of the identify for each \(\theta\). This globalizes as a multiple of \(|\cdot|_{\g_2^*}\) where \(|\cdot|_{\g_2^*}\) is the norm corresponding to the sphere \(\mathbb{S}^*\g_2\). This function is however not smooth on \(\g_2^*\) but only on \(\g_2^*\setminus 0\). This is why the function:
    \[\sigma' \colon (\xi,\eta) \mapsto -\frac{1}{2}|\xi|^2_{\g_1} -\left(1-\frac{n}{2}\right)|\eta|_{\g_2^*} \]
    is not smooth and thus does not define an element of \(\g_2^*\).

    This is what Theorem \ref{Theorem: Symbols group} circumvents. The workaround is that we can work with \(\theta\) fixed, use the functional calculus to get a Weyl operator acting as \((N+1)^{-1/2}\), and construct shifts by left multiplying with vectors from \(\g_1^{1,0}\). Families of such operators of order \(0\) do not define elements of \(\Sigma^0(G)\), even locally on \(\mathbb{S}^*\g_2\) as we lack some smoothness in the directions of \(\g_2^*\). However they are still continuous so they define elements of the completion \(\Sigma(G)\).
\end{rem}


\section{H-type groups}


In this section we investigate the special features of our construction for H-type groups. This is a special class of regular step-2 groups introduced by Kaplan \cite{Kaplan}. The motivation was that on these groups there is an explicit fundamental solution for the sublaplacian \(\sum_{i = 1}^{2n} X_i^2\) where \(X_1,\cdots,X_{2n}\) is a basis of \(\g_1\). These groups are endowed with a metric that is compatible with the symplectic structure in a way we will make explicit below. As such, they will be related to representations of Clifford algebras and thus become rather rigid. On the one hand, if we start with a H-type Lie algebras and consider a deformation within the class of H-type Lie algebras, the algebras stay in the same isomorphism class along the deformation, see Lemma \ref{Lemma: H-type rigidity} below. On the other hand, if we  allow the deformation to take place within the class of regular step-2 Lie algebras, it is very easy to get a non-H-type Lie algebra, see \cite{LevsteinTiraboschi}.

We approach these groups with the following problem. Our construction of Toeplitz algebra for a regular step-2 group \(G\) relied on the symplectic vector bundle \(\mathbb{S}^*\g_2\times \g_1\to \mathbb{S}^*\g_2\) (we have fixed a metric on \(\g_2\) and then see \(\mathbb{S}^*\g_2 \hookrightarrow \g_2^*\) as the unit sphere). Despite the bundle \(\mathbb{S}^*\g_2\times \g_1\) being trivial as a real vector bundle, it is not necessarily trivial as a symplectic vector bundle. The question is then, for which regular step-2 group can the symplectic vector bundle \(\mathbb{S}^*\g_2\times \g_1\) be trivialized ?

\begin{rem}
    All the reasonings below will implicitly use connectedness of the sphere \(\mathbb{S}^*\g_2\). For that we need to assume that \(\dim(\g_2)\geq 2\). The case where \(\dim(\g_2)=1\) corresponds to the Heisenberg groups. Indeed if \(\g\) is a regular step-2 Lie algebra with \(\g_2\) one dimensional, take any non-zero element \(\theta \in \g_2^*\). Then \(G\cong \Heis(\g_1,\omega_{\theta})\).
\end{rem}

Let \(N = \dim(\g_2)-1\geq 1\) be the dimension of the sphere, \(2n = \dim(\g_1)\). Real (oriented) vector bundles of rank \(2n\) over \(\mathbb{S}^*\g_2\) are classified by \(\pi_{N-1}(\mathrm{SO}(2n))\). If we choose a compatible complex structure on the symplectic bundle, the triviality of the symplectic bundle (as a symplectic vector bundle) is equivalent to the triviality of the complex bundle (as a complex vector bundle). The latter does not depend on the choice of compatible complex structure because all the resulting complex vector bundles are isomorphic. Complex vector bundles over \(\mathbb{S}^*\g_2\) are classified by \(\pi_{N-1}(\mathrm{U}(n))\). The exact sequence of homotopy groups of the fibration \( \mathrm{U}(n)\to \mathrm{SO}(2n) \to \faktor{\mathrm{SO}(2n)}{\mathrm{U}(n)}\) gives:
\[\xymatrix{\pi_N(\mathrm{SO}(2n)) \ar[r] & \pi_N\left(\faktor{\mathrm{SO}(2n)}{\mathrm{U}(n)}\right) \ar[dl]^{\delta} \\
 \pi_{N-1}(\mathrm{U}(n)) \ar[r] & \pi_{N-1}(\mathrm{SO}(2n)).}\] 
We have chosen a lift \(\xi \in \pi_{N-1}(\mathrm{U}(n))\) of our real vector bundle. By exactness we can write \(\xi = \delta(x)\) for some \(x \in \pi_N\left(\faktor{\mathrm{SO}(2n)}{\mathrm{U(n)}}\right)\). We have an explicit representative for this homotopy class. Indeed, remember that the set of symplectic forms on \(\g_1\) is homeomorphic to \(\faktor{\mathrm{SO}(2n)}{\mathrm{U(n)}}\). Then the map \(\omega \colon \theta \mapsto \omega_{\theta}\) defines a homotopy class \([\omega] \in \pi_N\left(\faktor{\mathrm{SO}(2n)}{\mathrm{U(n)}}\right)\) with \(\delta([\omega]) = \xi\) (this follows from the compatibility between the complex structure and the symplectic structure). Finally we want \(\xi  = 1\). By exactness this means that \([\omega]\) is the image of an element in \(\pi_N(\mathrm{SO}(2n))\). Consider the following diagram:

\[\xymatrix{\mathbb{S}^*\g_2 \ar[r]^(.35){\omega} \ar@{.>}[dr]_{\exists ? \Omega} & \faktor{\mathrm{SO}(2n)}{\mathrm{U}(n)} &  \\
 & \mathrm{O}(2n) \ar[u] & \mathrm{U}(n) \ar[l].}\]

The right hand side of the diagram is a fibration. We want a lift of \(\omega\) to \(\mathrm{SO}(2n)\). The sphere \(\mathbb{S}^*\g_2\) only has a cell in dimension \(0\) and \(N\). The obstruction to this lift lives in
\[H^N(\mathbb{S}^*\g_2,\pi_{N-1}(\mathrm{U}(n))) \cong \pi_{N-1}(\mathrm{U}(n)).\]
Indeed since \(\faktor{\mathrm{SO}(2n)}{\mathrm{U}(n)}\) is simply connected, the bundle of local coefficients is trivial, therefore the cohomology group is a genuine (singular) cohomology group. The identification then comes from the universal coefficients theorem. The problem is now that this class is precisely the class \(\xi\) above. Therefore we find that \(\xi\) is trivial if and only if it is trivial... so we don't get any new information. We can however give a geometric meaning to the lift \(\Omega\).

\begin{rem}
    If \(N-1< 2n\) and \(N\) is odd (so \(N-1\) even), we have \(\pi_{N-1}(U(n)) = 0\). Therefore there is no obstruction to the lift above and the symplectic bundle is trivial. However, we will see below that the groups satisfying this triviality condition are quite rigid, so not all values of \(n\) and \(N\) can be taken (compare with Remark \ref{Remark: dimension of H-type groups} below).
\end{rem}

The map \(\Omega\) takes values in \(\mathrm{SO}(2n)\) so we may fix a metric on \(\g_1\) so that all the occurrences of \(\mathrm{SO}(2n)\) above correspond to \(\mathrm{SO}(\g_1,\lag,\rag)\). For every \(\theta\) we have an orthogonal map \(\Omega_{\theta} \in \mathrm{SO}(\g_1,\lag,\rag)\). It defines a lift of \(\omega\) if and only if for every \(X,Y \in \g_1\), the map:
\[\theta \mapsto \omega_{\theta}(\Omega_{\theta}X,\Omega_{\theta}Y)\]
is constant. Let \(J_{\theta} \in \End(\g_1)\) be the complex structure compatible with \(\omega_{\theta}\) such that the corresponding metric is \(\lag,\rag\). Then the endomorphism \(\Omega_{\theta}^{-1}J_{\theta}\Omega_{\theta}\) does not depend on \(\theta \in \mathbb{S}^*\g_2\). Moreover this complex structure, and hence all the \(J_{\theta}\), is orthogonal for the inner product on \(\g_1\). This is precisely the definition of an H-type group.

\begin{mydef}[Kaplan] A step-2 Lie algebra is of H-type if it can be endowed with a metric such that \(\g_1\) and \(\g_2\) are orthogonal and such that the map \(J_z \in \End(\g_1), z \in \g_2^*\) defined by:
\[\forall X,Y \in \g_1,  \lag z , [X,Y]\rag = \lag J_z X, Y\rag,\]
is orthogonal whenever \(z\) has norm 1.
\end{mydef}

In this definition, we don't consider the metric to be part of the data. Indeed, it is proved in \cite{KaplanSubils} that two H-type groups with a compatible metric, that are isomorphic as groups, are also isomorphic in an isometric way.

\begin{thm}A regular step-2 Lie algebra is of H-type if and only if the symplectic bundle \(\mathbb{S}^*\g_2 \times \g_1 \to \mathbb{S}^*\g_2\) is trivial.\end{thm}
\begin{proof}
Thus far we have proved that if the bundle was trivial then the Lie algebra had to be of H-type. Conversely if \(\g\) is of H-type then the map \(\Omega \colon \theta \mapsto J_{\theta}\) lifts the map \(\omega\). Consequently the symplectic bundle is trivial.
\end{proof}

\begin{rem}
    We can include the Heisenberg groups in the previous theorem by saying that the bundle is trivial on each connected component of \(\mathbb{S}^*\g_2\).
\end{rem}

\begin{cor}If \(\g\) is an H-type Lie algebra, then trivializing the symplectic bundle gives a complex structure \(J_{cst}\in \End(\g_1)\). Then we have:
\[\T(\g, J_{cst}) \cong \CCC(\mathbb{S}^*\g_2) \otimes \T(\g_1^{1,0}).\]
Moreover \(\T_0(\g, J_{cst})\) is then the algebra of \(\T(\g_1^{1,0})\)-valued functions on \(\mathbb{S}^*\g_2\) such that the image in the quotient \(\CCC(\mathbb{S}^*\g_1)\) is constant.
\end{cor}

\begin{rem}
    If we take the Heisenberg group \(\Heis(V,\omega)\), then there is no complex structure compatible with both \(\omega\) and \(-\omega\). The algebras \(\T(V^{1,0})\) and \(\T(V^{0,1})\) are still isomorphic (since they are both isomorphic to \(\T(\CC^n)\) where \(\dim(V) = 2n\)). Under this non-natural isomorphism, we can also see \(\T(\heis(V,\omega))\) in the form as in the last Corollary.
\end{rem}

\begin{ex}
    Consider \((V,\omega)\) a complex symplectic space and let \(\heis_{\CC}(V,\omega)\) be the complex Lie algebra constructed from the abelian \(V\) with the 2-cocycle \(\omega\). If we see \((V,\omega)\) as the complexification of a real symplectic space \((V_0,\omega_0)\), then \(\heis_{\CC}(V,\omega)\) is the complexification of \(\heis(V_0,\omega_0)\). In particular we have now a 2-dimensional center. Trivializing the symplectic bundle, we get:
    \[\T(\heis_{\CC}(V,\omega),J_{cst}) = \CCC(\mathbb{S}^1) \otimes \T(V_0^{1,0}\oplus V_0^{0,1}).\]
\end{ex}

The H-type groups correspond to Clifford modules. Indeed, consider for \(z\in \g_2 \cong \g_2^*\), the endomorphism \(J_z \in \End(\g_1)\) characterized by the identity: \(\lag x,y\rag = \omega_z(J_zx,y)\). Then for \(x \in \g_1\), \(J_z(x)\) is the unique other element of \(\g_1\) satisfying \([J_z(x),x] = |x|^2z\). Some quick computations yield:
\[\forall x,y\in \g_1, \forall z,w \in \g_2, \lag J_z x,J_wy\rag +\lag J_wx,J_zy\rag = 2\lag x,y\rag\lag z,w\rag .\]
Now using anti-self adjointness, we obtain the Clifford relations:
\[J_zJ_w + J_wJ_z = -2\lag z,w\rag \Id.\]
Consequently, the data of a (metric) H-group is the data of a (unitary) \(\Cl(\g_2)\)-module structure on \(\g_1\). There is then a one to one correspondence between isomorphisms of Clifford modules and isomorphisms of Lie algebras being the identity on the center \(\g_2\). In general an isomorphism of Lie algebras doesn't have to be the identity on the center. Because of this, non-isomorphic Clifford modules could yield isomorphic H-type Lie algebras. It is shown in \cite{KaplanRicci} that, if we fix the dimension of the center, there is only one irreducible H-type Lie algebra with a center of that dimension up to isomorphism. Now if \(\mathfrak{v}_1,\mathfrak{v}_2\) are two complementary Clifford submodules of \(\g_1\) then we have \([\mathfrak{v}_1,\mathfrak{v}_2] = 0\). Indeed, assume \([x,y] = z \neq 0\) with \(x \in \mathfrak{v}_1, y\in \mathfrak{v}_2\). Then \(J_z(y) = |y|^2x\), contradiction. Therefore an H-type Lie algebra is completely classified, up to isomorphism, by the pair \((\dim(\g_2),\dim(\g_1))\). 

\begin{rem}\label{Remark: dimension of H-type groups}
    From the theory of Clifford modules, the dimension of \(\g_1\) has to be a non-zero multiple of the dimension of the irreducible Clifford modules of \(\Cl(\g_2)\). The latter only depends on the dimension of \(\g_2\)). If we write \(d_m\) for the dimension of that module, where \(m = \dim(\g_2)\), we have \(d_{m+8} = 16d_m\). Morevover the dimensions for small values of \(m\) is given by:
    \[d_1 = 2, \quad d_2=d_3=4, \quad d_4=d_5=d_6=d_7=8,\quad d_8=16.\]
    Therefore H-type groups can only have dimensions:
    \[(\dim(\g_1),\dim(\g_2)) = (kd_m,m),\quad m,k\in \mathbb{N}.\]
\end{rem}

\section{Polycontact manifolds}\label{Section: polycontact manifolds}


\begin{mydef}[van Erp]
    A polycontact structure on a smooth manifold \(M\) is a smooth distribution \(H\subset TM\) such that, if \(\theta \in \Gamma_{loc}(H^\perp)\) then \(\diff\theta_{|H \times H}\) is non-degenerate. The pair \((M,H)\) is called a polycontact manifold.
\end{mydef}

Given a smooth manifold \(M\) and a subbundle \(H\subset TM\), we can create a bundle of nilpotent groups the following way. Consider the vector bundle \(H \oplus \faktor{TM}{H}\). The curvature form of \(H\) is the bundle map \(H \times H \to \faktor{TM}{H}\) induced by the Lie bracket of vector fields:
\[(X_m,Y_m) \mapsto [X,Y]_m \mod H_m, \ m\in M.\]
Such a map is indeed well defined and does not depend on the choice of vector fields representing the elements \(X_m,Y_m\in H_m\). This form endows \(H \oplus \faktor{TM}{H}\) with a Lie algebra structure on each fiber by taking the Lie bracket with any element of \(\faktor{TM}{H}\) to be zero. This Lie algebra structure depends smoothly on the base-point, but is not necessarily locally trivial.

The Lie algebras being nilpotent, we can integrate them all at once using the Baker-Campbell-Hausdorff formula. This gives a smooth family \(T_HM\) of nilpotent groups on \(M\) called the osculating groups.

\begin{prop}
    Let \(M\) be a manifold endowed with a smooth distribution \(H\subset TM\). The following are equivalent:
    \begin{enumerate}
        \item \((M,H)\) is a polycontact manifold.
        \item For every \(m\in M, \theta \in H_m^{\perp} \setminus 0\), the form \(\omega_{\theta} \in \Lambda^2H_m^*\) defined by composing the curvature 2-form with \(\theta\) is non-degenerate.
        \item All osculating groups are regular 2-step nilpotent groups. 
    \end{enumerate}
\end{prop}

\begin{proof}
    Let \(X,Y\in \Gamma_{loc}(H)\) and \(\theta \in \Gamma_{loc}(H^\perp)\). We have:
    \[\diff\theta(X,Y) = X\cdot\theta(Y) - Y\cdot \theta(X) - \theta([X,Y]) = -\theta([X,Y]),\]
    thus the two first conditions are equivalent. The equivalence between the last two conditions is Proposition \ref{Proposition:Regular 2-step Conditions}.
\end{proof}

\begin{ex}
    Contact manifolds, quaternionic and octonionic contact manifolds \cite{BiquardQuaternionicContact} are all examples of polycontact manifold.
    
    A regular 2-step nilpotent group, with the left invariant distribution induced by the filtration of its Lie algebra is a polycontact manifold. If \(\Gamma \subset G\) is a discrete subgroup, the manifold \(\faktor{G}{\Gamma}\) is a polycontact manifold. If \(\Gamma\) is cocompact we then have a compact polycontact manifold.
\end{ex}

\begin{rem}
    In all the previous examples, the osculating groups are the same no matter the point on a manifold. This is because these examples are either constructed from the group, or are particular cases of H-type manifolds, as defined below, where this phenomenon is to be expected.

    It is however not the general case. It is shown in \cite{LevsteinTiraboschi} how a H-type group can be deformed into a family of non-isomorphic regular 2-step nilpotent groups. 
\end{rem}

\section{Fields of Toeplitz algebras}\label{Section:Fields Toeplitz Algebras polycontact}


Let \((M,H)\) be a polycontact manifold. From the family of osculating groups \(T_HM\) we obtain a \(C^*\)-algebra \(C^*(T_HM)\). This \(C^*\)-algebra is a continuous field of \(C^*\)-algebras over \(M\), with fiber at \(x\in M\) equal to the group \(C^*\)-algebra \(C^*(T_{H,x}M)\). The spectrum of \(C^* (T_HM)\) is homeomorphic to \(\faktor{\ttt_HM^*}{\Ad^*(T_HM)}\). We again have a dichotomy between the flat orbits and points. The flat orbits are parameterized by \(H^{\perp}\setminus M\) and correspond to infinite dimensional representations. They provide an ideal of \(C^*(T_HM)\) isomorphic to \(\CCC_0(H^{\perp}\setminus M, \mathcal{K})\). The quotient corresponds to the characters, i.e. fixed points for the coadjoint action, parameterized by \(H^*\).

\begin{prop}\label{Prop: Extension CStar Osculating}
    Let \((M,H)\) be a polycontact manifold, we have the exact sequence:
    \[\xymatrix{0 \ar[r] & \CCC_0(H^{\perp}\setminus M, \mathcal{K}) \ar[r] & C^*(T_HM) \ar[r] & \CCC_0(H^*) \ar[r] & 0.}\]
\end{prop}

\begin{proof}
    The maps are well defined globally. The sequence is then exact on each fiber by Corollary \ref{Corollary: Extension CStar Group}, so it is itself exact. Real codimension 1 submanifolds in hypercomplex spaces (with strong pseudoconvexity assumptions) are an example of polycontact manifold where the osculating groups are not necessarily all the same \cite{VanErpPolycontact}.
\end{proof}

Using this proposition, we define \(C^*_0(T_HM)\) to be the preimage by the quotient map of \(\CCC_0(H^*\setminus M)\). It is the algebra of sections of the continuous field of \(C^*\)-algebras formed by the \((C^*_0(T_{H,x}M))_{x\in M}\) (seen as a subfield of the one underlying \(C^*(T_HM)\). We want to express \(C^*_0(T_HM)\) as a crossed product, gluing our previous construction with all the osculating groups at the same time.

Consider for each \((x,\theta)\in H^{\perp}\setminus M\), the form \(\omega_{\theta}\in \Lambda^2H_x\) obtained by composition with the curvature form (i.e. the Lie bracket of \(\ttt_HM\)). Since we are on a polycontact manifold, \(\omega_{\theta}\) is a symplectic form. We therefore get a symplectic bundle \(H \times_M (H^{\perp}\setminus M) \to H^{\perp}\setminus M\). We may fix a metric on \(\faktor{TM}{H}\), hence on \(H^{\perp}\), and restrict the bundle to \(H \times_M \mathbb{S}H^{\perp} \to \mathbb{S}H^{\perp}\). Denote by \(\pi \colon \mathbb{S}H^{\perp} \to M\) the bundle map, so that the previous bundle becomes \(\pi^*H\). For any given \(x\in M\), we can restrict the bundle to \(H_x \times \mathbb{S}H^{\perp}_x \to  \mathbb{S}H^{\perp}_x\), which is exactly the symplectic bundle for the osculating group \(T_{H,x}M\).

Now that we have a symplectic bundle, we may choose a compatible complex structure \(J \in \End(\pi^*H)\) and the corresponding \(C^*\)-algebra of sections of the Toeplitz bundle \(\T(\ttt_HM,J)\). The symplectic bundle restricts over each \(\mathbb{S}H_x^{\perp}, x\in M\) to the one of the osculating group. We can thus see \(\T(\ttt_HM,J)\) as the algebra of sections of a continuous field of \(C^*\)-algebras over \(M\), whose fiber at \(x \in M\) is:
\[\T(\ttt_HM,J)_x = \T(\ttt_{H,x}M,J_x),\] 
where \(J_x = J_{|H_x \times \mathbb{S}H_x^{\perp}}\).
We denote by \(\T_0(\ttt_HM,J)\) the subalgebra of \(\T(\ttt_HM,J)\) consisting of sections whose value at \(x \in M\) is in \(\T_0(\ttt_{H,x}M,J_x)\). We have the exact sequence:
\[\xymatrix{0 \ar[r] & \CCC_0(M,\mathcal{K}) \ar[r] & \T_0(\ttt_HM,J) \ar[r] & \CCC_0(\mathbb{S}H^*) \ar[r] & 0.}\]

\begin{thm}\label{Theorem: Symbols Osculating}
	There is an isomorphism of \(C^*\)-algebras:
	\[\T_0(\ttt_HM,J) \cong \Sigma(T_HM).\]
\end{thm}

\begin{proof}
The proof is similar as for a single group. The algebra \(\Sigma(T_HM)\) is the subalgebra of sections of Weyl operators \(\mathcal{W}^0(H_x^*,\omega^*_{\theta}), (x,\theta)\in \mathbb{S}H^{\perp}\) such that the corresponding field of principal symbols is constant on each \(\mathbb{S}H_x^{\perp}\). This is then the same as Theorem \ref{Theorem: Symbols group}.
\end{proof}

Like on a group, we can define the number operator. The algebra \(\T(\ttt_HM,J)\) acts on the Hilbert bundle of symmetric Fock spaces \(\mathcal{F}^+((\pi^*H,J)^{1,0})\). The number operator \(N\) acts on \(\Sym^k((\pi^*H,J)^{1,0})\) by \(k\Id\) and defines this way an unbounded operator. Its complex powers define a strongly continuous group of bounded operators. As for the group, conjugation by \(N^{\ii t/2}, t\in \R\) defines a group of automorphisms of \(\T(\ttt_HM,J)\) and \(\T_0(\ttt_HM,J)\).

\begin{thm}\label{Theorem: Toeplitz to Osculating}
	Let \((M,H)\) be a compact polycontact manifold. Let \(J\in \End(\pi^*H)\) be a compatible complex structure on the symplectic bundle. Then we have an \(\R^*_+\)-equivariant isomorphism:
	\[\T_0(\ttt_HM,J) \rtimes \R \cong C^*_0(T_HM).\]
\end{thm}

\begin{proof}
As for the group case, we identify \(C^*_0(T_HM)\) with a subalgebra of sections of the bundle \((I_{+,0}(H_x,\omega_{\theta}))_{(x,\theta)\in \mathbb{S}H_x^{\perp}}\). These sections must satisfy that, for \(x\in M\) fixed, their restriction at \(0\) in \(I_{+,0}(H_x,\omega_{\theta}), \theta \in \mathbb{S}H_x^{\perp}\) is constant as a family of elements of \(\CCC_0(H_x^*\setminus 0)\).

Let \(\widetilde{C}^*_0(T_HM)\) be the algebra of (all the) sections of the bundle of \(C^*\)-algebras \((I_{+,0}(H_x,\omega_{\theta}))_{(x,\theta)\in \mathbb{S}H_x^{\perp}}\). We want to use the same arguments as in \ref{Theorem: Toeplitz Field crossed product} to show an isomorphism \(\T(\ttt_HM,J)\rtimes \R \isomto \widetilde{C}^*_0(T_HM)\). This is similar to \cite[Theorem 7.7]{CrenToeplitz}.  
\end{proof}

\begin{rem}
	The construction of the map \(\T(V^{1,0}) \rtimes \R \to I_{+,0}(V,\omega)\) amounts to the existence of an operator \(\Delta \in \Sigma^2(\Heis(V,\omega))\) such that \(\pi_{\lambda}(\Delta) = \lambda^{2}(N+1)\) for every positive \(\lambda\) and then send \(S_X\) to \(-\ii X \Delta^{-1/2}\) (compare with Corollary \ref{Corollary: Toeplitz isom Weyl}). In our situation, we would need such a \(\Delta\) for each point \((x,\theta)\). In case \(M\) is not compact, we also need to control the norm of these type of elements at infinity and it is not clear it should decay at infinity. This is the reason we only state the previous theorem when \(M\) is compact (similarly as in the previous \cite{CrenToeplitz}).
\end{rem}

\section{The case of H-type manifolds}\label{Section: H-type manifolds}


\begin{mydef}
    A polycontact manifold \((M,H)\) is called a H-type manifold if all its osculating groups are H-type groups.
\end{mydef}

A geometric characterization of these manifolds is given by the following:

\begin{prop}[\cite{KaplanSubilsEquivalence} Proposition 5.1]
    A distribution \(H \subset TM\) makes \(M\) a H-type manifold if and only if it admits a compatible (sub)conformal structure.
\end{prop}

\begin{ex}
    Contact, quaternionic contact and octonionic contact manifolds are all examples of H-type manifolds.
\end{ex}

\begin{lem}\label{Lemma: H-type rigidity}
    If \((M,H)\) is a H-type manifold then all its osculating groups are mutually isomorphic.
\end{lem}

\begin{proof}
    The isomorphism class of an H-type Lie algebra is completely characterized by the dimension and codimension of its center. The distribution \(H\subset TM\) having constant rank (and co-rank), if the osculating groups are all H-type, they have to be mutually isomorphic.
\end{proof}

\begin{thm}
    If \((M,H)\) is a H-type manifold then \(T_HM\) is a locally trivial bundle of Lie groups.
\end{thm}

\begin{proof}
    This is a consequence of a more general fact, that a (smooth) family of Lie algebras that form a (locally trivial) vector bundle and where all the fibers are isomorphic has to be a locally trivial bundle of Lie algebras \cite{Kiranagi}.
\end{proof}

\begin{rem}
    In that context where the osculating groups form a locally trivial group bundle, we should probably even have local diffeomorphisms to the typical fiber, in a way that preserves the respective filtrations. This would be in contrast with some rigidity results known in quaternionic contact geometry. It is indeed proved in \cite{IvanovVassilevLocalInvariantQuaternionicContact} that for those manifolds, the curvature of the Biquard connection \cite{BiquardQuaternionicContact} prevents a quaternionic contact manifold to be isomorphic to the local model \(\mathbb{H}^n \times \Im (\mathbb{H})\). This notion of isomorphism however requires to preserve the conformal class on the metric on the distribution, which is stronger than the local isomorphism we want.
    Compare also this claim with \cite{MontgomeryGenericDistributions} where it's proved that for generic distributions, finding a local frame that generates a finite dimensional Lie algebra (which is equivalent to our claim) is only possible in very specific cases. However here we are not asking for generic distributions since requiring to have the same osculating group everywhere (at least locally) is quite rigid.
\end{rem}



\begin{cor}
    If \((M,H)\) is a H-type manifold with typical osculating group \(G\), then the algebras \(\Sigma(T_HM), C^*(T_HM), C^*_0(T_HM)\) are algebras of sections of locally trivial bundles of \(C^*\)-algebras with respective fiber \(\Sigma(G), C^*(G), C^*_0(G)\).
\end{cor}

\begin{cor}
    Let \((M,H)\) be a H-type manifold with typical osculating group \(G\). Choose a complex structure \(J \in \End(\pi^*(H))\) compatible with the symplectic form and that restricts to \(J_{cst}\) on each fiber, where \(\pi \colon \mathbb{S}H^{\perp} \to M\) is the projection. then the bundle of Toeplitz algebras \(\T(\pi^*H,J) \to \mathbb{S}H^{\perp}\) can be trivialized over subsets of the form \(\pi^{-1}(U)\).  More precisely, locally over \(U \subset M\) where we trivialize the bundle \(\mathbb{S}H^{\perp}\), we get 
    \[\T(\pi^*H,J)_{|\pi^{-1}(U)} \cong \CCC_0(U \times \mathbb{S}^*\g_2)\otimes \T_0(\g,J_{cst}).\]
\end{cor}

\bibliographystyle{plain}
\bibliography{Ref}
\end{document}